\documentclass[12pt,reqno]{article}

\usepackage[usenames]{color}
\usepackage{amssymb}
\usepackage{graphicx}
\usepackage{amscd}
\usepackage{float}

\usepackage[colorlinks=true,
linkcolor=webgreen,
filecolor=webbrown,
citecolor=webgreen]{hyperref}

\definecolor{webgreen}{rgb}{0,.5,0}
\definecolor{webbrown}{rgb}{.6,0,0}

\usepackage{color}
\usepackage{float}

\usepackage{amsmath,amssymb}
\usepackage{amsthm}
\usepackage{amsfonts}
\usepackage{latexsym}
\usepackage{epsf}

\newcommand{\seqnum}[1]{\href{http://oeis.org/#1}{\underline{#1}}}

\newtheorem{problem}{Problem}

\def\hfl#1#2#3{\smash{\mathop{\hbox to#3{\rightarrowfill}}\limits
^{\scriptstyle#1}_{\scriptstyle#2}}}

\def\vfl#1#2#3{\llap{$\scriptstyle #1$}
\left\downarrow\vbox to#3{}\right.\rlap{$\scriptstyle #2$}}

\begin{document}
\begin{center}
\epsfxsize=4in
\end{center}

\rightline{May 27, 2011}

\begin{center}
\vskip 1cm{\LARGE\bf Enumeration of Standard Puzzles}
\vskip 1cm
\large
Guo-Niu Han\\
I.R.M.A., UMR 7501\\
Universit\'e de Strasbourg et CNRS\\
7 rue Ren\'e-Descartes, 67084 Strasbourg, France\\
\href{mailto:guoniu.han@unistra.fr}{guoniu.han@unistra.fr}\\
\end{center}

\vskip .2 in

\begin{abstract}
We introduce a large family of combinatorial 
objects, called standard puzzles, defined by very simple rules.
We focus on the standard puzzles for which the enumeration problems
can be solved by explicit formulas or by classical numbers,
such as binomial coefficients, Fibonacci numbers, tangent numbers, 
Catalan numbers,	$\ldots$
\end{abstract}


\theoremstyle{plain}
\newtheorem{theorem}{Theorem}
\newtheorem{corollary}[theorem]{Corollary}
\newtheorem{lemma}[theorem]{Lemma}
\newtheorem{proposition}[theorem]{Proposition}

\theoremstyle{definition}
\newtheorem{definition}[theorem]{Definition}
\newtheorem{example}[theorem]{Example}
\newtheorem{conjecture}[theorem]{Conjecture}

\theoremstyle{remark}
\newtheorem{remark}[theorem]{Remark}

\newcommand{\NN}{\mathbb{N}}
\newcommand{\ZZ}{\mathbb{Z}}
\newcommand{\QQ}{\mathbb{Q}}
\newcommand{\im}{\mbox{\upshape Im}}

\newcommand{\stirlings}[2]{\genfrac\{\}{0pt}{}{#1}{#2}}
\newcommand{\stirlingf}[2]{\genfrac[]{0pt}{}{#1}{#2}}

\newtheorem{Theorem}{Theorem}[section]
\newtheorem{Proposition}[Theorem]{Proposition}
\newtheorem{Corollary}[Theorem]{Corollary}
\theoremstyle{definition}
\newtheorem{Example}[Theorem]{Example}
\newtheorem{Remark}[Theorem]{Remark}
\newtheorem{Problem}[Theorem]{Problem}

\newcommand{\piece}[4]{\bigl[\mbox{\scriptsize ${\begin{matrix}#4\ #3\\ #1\ #2\\ \end{matrix}}$}\bigr]}
\newcommand{\smatrix}[1]{\mbox{\scriptsize $\begin{matrix}#1 \end{matrix}$}}
\newcommand{\setP}{\mathcal{P}}
\newcommand{\setpieces}[2]{\bigl\{#1\bigr\}^{#2}}
\newcommand{\cardpieces}[2]{\bigl|#1\bigr|^{#2}}

\section{Introduction} 
We introduce a large family of combinatorial objects, 
called {\it standard puzzles}, defined by very simple rules,
and study their enumeration problems. The topic of the paper may be classified
as belonging to {\it Enumerative Combinatorics}, since several of those 
standard puzzles 
can be solved by explicit formulas or by using classical numbers,
such as binomial coefficients, Fibonacci numbers, tangent numbers, 
Catalan numbers, $\ldots$

\vspace{\medskipamount}

The general definition of a standard puzzle is inspired by the following
classical topics in Enumerative Combinatorics:
(1) polyominoes \cite{Wiki_Polyomino} ;
(2) standard Young tableaux \cite{Wiki_YoungTab} ;
(3) permutation patterns \cite{Wiki_Pattern} ; 
(4) doubloons, which were introduced recently \cite{GZ,FHa, FHb}.

\vspace{\medskipamount}

We do not pretend to establish a general principle
that will make the enumeration of
all those standard puzzles possible. 
We will only provide a large list of standard puzzles with their 
first values, and the OEIS outputs whenever the sequence is already listed in 
the On-Line Encyclopedia of Integer Sequences \cite{Sloane}.  
Up to order 4 there are 114 sequences identified in OEIS,
meaning that those sequences have already been found in various analytical, 
arithmetical or combinatorial contexts and can also be derived in our 
puzzle model,
but 1339 are still unknown,
that is, do have puzzle descriptions, 
but have not been encountered in other contexts.
\section{Definition of standard puzzle} 
A {\it piece} is a square having four numbers, called {\it labels},
written on its corners.
A {\it puzzle} is a connected arrangement of pieces in the $\ZZ\times\ZZ$-plane
such that the joining corners of all the pieces have the same labels (see Fig. \ref{fig.puzzle}).
Pieces and puzzles can be translated, but can be neither rotated, 
nor reflected.
The {\it shape} of each puzzle is a polyomino obtained from the puzzle
by removing the labels.
%
\long\def\pythonbegin#1\pythonend{}
\pythonbegin
beginfig(1, "1.8mm");  
setLegoUnit([6,6])
#u_showgrid([0,0], [8,8])  # show grid
r=0.18 # dot radius

def DotGrayLabel(pt, radius, lab):
	circle(pt, radius)
	label(pt, "{$\scriptstyle"+str(lab) + "$}")

def AddPath(p, dir):
	p.append([p[len(p)-1][0]+ dir[0], p[len(p)-1][1]+dir[1]])

def DrawSquare(p, labs):
	pa=[p]
	(N,S,E,O)=([0,1], [0,-1], [1,0], [-1,0])
	[AddPath(pa, z) for z in [E,N,O,S]]
	[DotGrayLabel(pa[z], r, labs[z]) for z in range(len(pa)-1) if labs[z]>0] 
	print "pickup pencircle scaled 0.6 pt;"
	[arc(pa[z], pa[z+1], marge=2*r, h=0) for z in range(len(pa)-1)]

DrawSquare([4.1,3], [2,9, 3,4])

p1=",".join([str(z) + "u" for z in	lego_pt([4.5, 2.5])])
p2=",".join([str(z) + "u" for z in	lego_pt([2.5, 1.5])])
print "HIArrowArc((

DrawSquare([0,2], [0,0,7,6])
DrawSquare([0,1], [0,0,5,8])
DrawSquare([1,1], [0,0,4,0])
DrawSquare([-1,0], [3,0,0,9])
DrawSquare([0,0], [5,3,7,4])
DrawSquare([1,0], [0,9,2,0])
DrawSquare([2,0], [0,1,9,0])
DrawSquare([3,0], [0,2,2,0])
DrawSquare([0,-1],[4,2,0,0])

draw_options("withpen pencircle scaled 0.4pt dashed evenly scaled 1")
arc([2,2], [3,2], marge=2*r, h=0)
arc([3,1], [3,2], marge=2*r, h=0)
line([2.9,2], [3,2])
line([3,1.9], [3,2])

endfig();
\pythonend
\begin{figure}[ht]
\centering
\begin{center}
\includegraphics{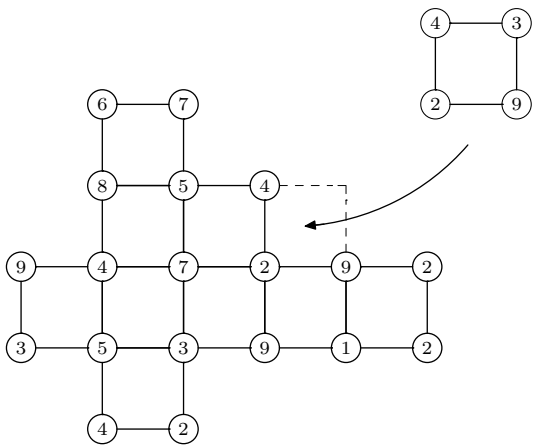}
\end{center}
\vspace{-12pt}
\caption{Adding a piece to a puzzle}
\label{fig.puzzle}
\end{figure}

A {\it standard puzzle} of shape $\lambda$ is a puzzle 
such that the multi-set of all its labels is simply 
$\{1,2,\ldots, m\}$. This implies that $m$ is the number of vertices of 
the polyomino $\lambda$. 
In particular, the four labels of a piece occurring in a standard puzzle are all distinct.
Replacing the four labels by $\{1,2,3,4\}$ respecting the label ordering
yields a {\it standard piece}. This operation is called a 
{\it reduction} and will be donoted by $\Omega$
(see Fig. \ref{fig.reduction}). Two pieces are said to be {\it identical} if 
they have the same reductions.
In the rest of the paper, puzzle means standard puzzle and piece means
standard piece.  We then have only twenty-four different pieces which are listed and coded in 
Table \ref{tab.codes}.  Note that the letters ``I" and  ``O" have not even used.
%
%
\long\def\pythonbegin#1\pythonend{}
\pythonbegin
beginfig(2, "1.8mm");  
setLegoUnit([6,6])
#u_showgrid([0,0], [8,8])  # show grid
r=0.2  # dot radius

DrawSquare([1,0], [15,3,11,4])
label([3.5,0.9], "$\Omega$")
label([3.5,0.5], "$\longmapsto$")
DrawSquare([4.7,0], [4,1,3,2])

endfig();
\pythonend
\begin{figure}[ht]
\begin{center}
\includegraphics{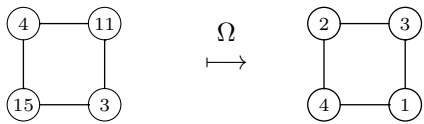}
\end{center}
\vspace{-12pt}
\caption{Reduction of a piece}
\label{fig.reduction}
\end{figure}
\begin{table}[ht]
$$
\begin{matrix}
&A =\piece 1234 &B =\piece 1243 &C =\piece 1324 &D =\piece 1342 &E =\piece 1423 &F =\piece 1432\\
\noalign{\vspace{\smallskipamount}
}
&G =\piece 2134 &H =\piece 2143 &J =\piece 2314 &K =\piece 2341 &L =\piece 2413 &M =\piece 2431\\
\noalign{\vspace{\smallskipamount}
}
&N =\piece 3124 &P =\piece 3142 &Q =\piece 3214 &R =\piece 3241 &S =\piece 3412 &T =\piece 3421\\
\noalign{\vspace{\smallskipamount}
}
&U =\piece 4123 &V =\piece 4132 &W =\piece 4213 &X =\piece 4231 &Y =\piece 4312 &Z =\piece 4321\\
\end{matrix}
$$
\caption{The twenty-four pieces and its codes}
\label{tab.codes}
\end{table}

Given a standard puzzle $\alpha$, a set $\setP$ of standard pieces is called
{\it support} of the puzzle~$\alpha$ if the reduction of each piece
occurring in the puzzle $\alpha$ is an element of $\setP$. Hence, the minimal
support of $\alpha$ is the set of all different pieces of $\alpha$ 
after reduction. For example,
the standard puzzle given in Fig. \ref{fig.standard.puzzle} 
contains seven pieces, but only four of them are different
$\piece 4132, \piece 1423, \piece 3241, \piece 4123$. 
The minimal support is the set of those four pieces and
each set of pieces
containing those four pieces is a support of the puzzle.
The main problem discussed in the paper is the following.
%
\long\def\pythonbegin#1\pythonend{}
\pythonbegin
beginfig(3, "1.8mm");  
setLegoUnit([6,6])
#u_showgrid([0,0], [8,8])  # show grid
r=0.2  # dot radius

DrawSquare([0,2], [0,0,7,6])
DrawSquare([0,1], [0,0,5,8])
DrawSquare([1,1], [0,0,14,0])
DrawSquare([0,0], [15,3,11,4])
DrawSquare([1,0], [0,12,10,0])
DrawSquare([2,0], [0,1,9,0])
DrawSquare([3,0], [0,13,2,0])

endfig();
\pythonend
\begin{figure}[ht]
\centering
\begin{center}
\includegraphics{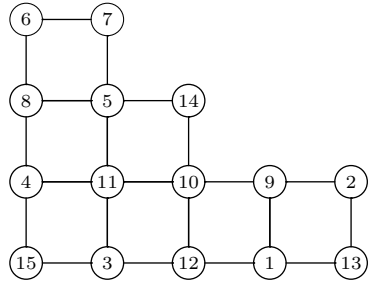}
\end{center}
\vspace{-12pt}
\caption{A standard puzzle}
\label{fig.standard.puzzle}
\end{figure}

\begin{problem}
\label{problem}
Given a set of pieces $\setP$ and a polyomino $\lambda$, 
count the number of standard puzzles of shape $\lambda$ whose support is equal  
to $\setP$. 
\end{problem}
\vspace{\medskipamount}

If the support is
$\bigl\{ \piece 1243, \piece 1342 \bigr\}$ and the shape is 
a partition, then a standard puzzle is just a standard Young 
tableau \cite{Wiki_YoungTab}. Hence, the number of standard puzzles is the number of
standard Young tableaux of fixed shape (See Fig.~\ref{fig.BD} in Section \ref{sec.example}). Problem \ref{problem} 
can be solved by using the famous hook length formula \cite{Hook}.

However, the problem is very hard to solve in general. In the rest of the paper
we focus our attention to a very special shape, namely, the 
$2\times n$ matrix.
For convenience puzzles of shape $2\times n$ will be represented by two-row matrices. For example $\bigl[\smatrix{5784\cr1236\cr}\bigr]$ stands for
%
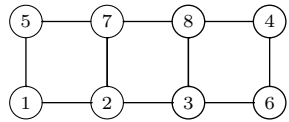
\long\def\pythonbegin#1\pythonend{}
\pythonbegin
beginfig(99, "1.8mm");  
setLegoUnit([6,6])
#u_showgrid([0,0], [8,8])  # show grid
r=0.2  # dot radius

DrawSquare([0,0], [1,2,7,5])
DrawSquare([1,0], [0,3,8,0])
DrawSquare([2,0], [0,6,4,0])

endfig();
\pythonend
\begin{figure}[H]
\centering
\begin{center}
\includegraphics{\jobname99.eps}
\end{center}
\end{figure}
\vspace{-14pt}
\noindent
In this case the set of all puzzles made by using the support $\setP$ is denoted
by $\setP^n$. For example, take 
$\setP=BC=\{B,C\}=\bigl\{\piece 1243, \piece 1324\bigr\}$.
Then
\begin{eqnarray*}
BC^2 &= &
\bigl\{\bigl[\smatrix{34\cr12\cr}\bigr],\bigl[\smatrix{42\cr13\cr}\bigr]\bigr\}\cr
BC^3 &=  &
\bigl\{\bigl[\smatrix{563\cr124\cr}\bigr],\bigl[\smatrix{463\cr125\cr}\bigr],\bigl[\smatrix{456\cr123\cr}\bigr],\bigl[\smatrix{364\cr125\cr}\bigr],\bigl[\smatrix{356\cr124\cr}\bigr]\bigr\}\cr
BC^4 &= &
\bigl\{
\bigl[\smatrix{6784\cr1235\cr}\bigr],
	\bigl[\smatrix{5784\cr1236\cr}\bigr],
	\bigl[\smatrix{5684\cr1237\cr}\bigr],
	\bigl[\smatrix{5678\cr1234\cr}\bigr],
	\bigl[\smatrix{4785\cr1236\cr}\bigr],
	\bigl[\smatrix{4685\cr1237\cr}\bigr],
\bigl[\smatrix{4678\cr1235\cr}\bigr],\cr
& &\quad \bigl[\smatrix{4586\cr1237\cr}\bigr],
\bigl[\smatrix{4578\cr1236\cr}\bigr],
\bigl[\smatrix{3785\cr1246\cr}\bigr],
\bigl[\smatrix{3685\cr1247\cr}\bigr],
\bigl[\smatrix{3678\cr1245\cr}\bigr],
\bigl[\smatrix{3586\cr1247\cr}\bigr],
\bigl[\smatrix{3578\cr1246\cr}\bigr]\bigr
\}\cr
\end{eqnarray*}
The sequence  $(|BC^n|)_n=(|BC^2|, |BC^3|, |BC^4|, \ldots)$, equal to 
$$(2, 5, 14, 42, 132, 429, 1430, 4862, 16796, 58786, 208012,
\ldots ), $$
is the well-known sequence of the Catalan numbers. 
However, many sequences defined by other supports 
are not identified in OEIS. For example,
with $\setP=CK=\bigl\{\piece 1324, \piece 2314\bigr\}$ we have
\begin{eqnarray*}
CK^2&=&\bigl\{\bigl[\smatrix{42\cr13\cr}\bigr],\bigl[\smatrix{14\cr23\cr}\bigr]\bigr\}\cr
CK^3&=&
\bigl\{\bigl[\smatrix{625\cr134\cr}\bigr],\bigl[\smatrix{526\cr134\cr}\bigr],\bigl[\smatrix{426\cr135\cr}\bigr],\bigl[\smatrix{164\cr235\cr}\bigr]\bigr\}\cr
CK^4&=&
\bigl\{\bigl[\smatrix{8275\cr1346\cr}\bigr],\bigl[\smatrix{7285\cr1346\cr}\bigr],\bigl[\smatrix{6285\cr1347\cr}\bigr],\bigl[\smatrix{5286\cr1347\cr}\bigr],\bigl[\smatrix{4286\cr1357\cr}\bigr],\bigl[\smatrix{1847\cr2356\cr}\bigr],\bigl[\smatrix{1748\cr2356\cr}\bigr],\bigl[\smatrix{1648\cr2357\cr}\bigr]\bigr\}\cr
\end{eqnarray*}
The sequence $(|CK^n|)_n$ is
$$(2, 4, 8, 26, 66, 276, 816, 4050,   13410, 75780, 274680  \ldots )$$
which is not in OEIS.

\section{The two-line matrix shape} 
When the shape is a $2\times n$ matrix, the enumeration problems for
the supports $\setP$ and $\setP'$ are equivalent, denoted by
$\setP\equiv \setP'$, if $\setP'$ is
obtained from $\setP$ by applying the following
basic transformations one or more times:
\vspace{\smallskipamount}

(T1) exchanging left column and right column in every piece;
\vspace{\smallskipamount}

(T2) exchanging top row and bottom row in every piece;
\vspace{\smallskipamount}

(T3) replacing each label $a$ by $(5-a)$ in every piece.
\vspace{\medskipamount}

For example, the support 
$\{\piece 1234,\piece 1243\}$ is equivalent 
to $\{\piece 2143,\piece 2134\}$ by (T1), 
to $\{\piece 4321,\piece 3421\}$ by (T2) 
and
to $\{\piece 4321,\piece 4312\}$ by (T3).
The enumeration problems for the following supports 
$\setP$ are all equivalent.
$$
\begin{matrix}
&\{\piece 1234,\piece 1243\} \equiv  \{\piece 3421,\piece 4321\}\equiv
\{\piece 2134,\piece 2143\} \equiv  \{\piece 3412,\piece 4312\}  \cr
\noalign{\vspace{\smallskipamount} }
\equiv&\{\piece 4312,\piece 4321\} \equiv  \{\piece 1234,\piece 2134\}\equiv
\{\piece 3412,\piece 3421\} \equiv  \{\piece 1243,\piece 2143\}. \cr
\end{matrix}
$$
\vspace{\medskipamount}

A support $\setP$ is said to be {\it connected} if for every pair of pieces
from $\setP$, there is a puzzle containing those two pieces.
The non-connected supports
$\setP$ can be fully characterized, and have one of the following three forms:
{
\def\piecetype#1#2{
\Bigl[\begin{matrix}
	\cdot & \cdot \cr
\noalign{\vskip -3pt}
#1  & #2 \cr
\noalign{\vskip -3pt}
	\cdot & \cdot \cr
		\end{matrix}
\Bigr]
}
\begin{eqnarray*}
(a)&\left\{\piecetype{\wedge}{\wedge},\ 
\piecetype{\wedge}{\wedge},
\cdots,
\piecetype{\wedge}{\wedge};\quad
\piecetype{\vee}{\vee},\ 
\piecetype{\vee}{\vee},
\cdots ,
\piecetype{\vee}{\vee}\right\}&
\cr
(b)&\left\{\piecetype{\wedge}{\vee};\quad
\piecetype{\wedge}{\vee};\quad
\cdots\ ;
\piecetype{\wedge}{\vee}\right\}&
\cr
(c)&\left\{\piecetype{\vee}{\wedge};\quad
\piecetype{\vee}{\wedge};\quad
\cdots \ ;
\piecetype{\vee}{\wedge}\right\}&
\end{eqnarray*}
} 

For example, the support 
$\{ \piece 1234, \piece 1243;  \piece 4321\} $
is not-connected (type (a)). It contains two connected components.
For studying the enumeration problem we need only consider
the connected supports.
After reduction by (T1-T3), the number of possible connected 
supports $\setP$ is shown in the following
table
$$
\begin{matrix}
|\setP|&: &	 1 &2 &3 &4 & 5    & 6&\cdots \cr
\#\{\setP\}   &:   &  6 & 37 & 259 & 1391 & 5460 & ?  &\cdots \cr 
\end{matrix}
$$
The list of all those sets $\setP$ with the sequences can be found in Section \ref{sec.dict}.

\begin{problem}
What is the sequence of the numbers of possible connected supports $\#\{\setP\}$ ?
\end{problem}

\vspace{\medskipamount}

Fix a support $\setP$. 
Let $f_n[\smatrix{X\cr Y\cr}]$
be the number of all puzzles of shape $2\times n$ 
such that the rightmost colunm is $[\smatrix{X\cr Y\cr}]$.
Then, the number $f_n$ of all puzzles of shape $2\times n$ is then
\begin{equation}\label{eq.sum}
f_n = \sum_{X,Y} f_n[\smatrix{X\cr Y\cr}],
\end{equation}
where the sum ranges over all ordered pairs of $[2n]\times [2n]$.
\vspace{\medskipamount}

The {\it inverse reduction} $x'=\Omega^{-1}(x; X, Y)$ of $x$ by $\{X,Y\}$ is defined by
$$
x'=
\begin{cases}
x;   &\text{ if $x\leq a-1$} \cr
x+1; &\text{ if $a\leq x \leq b-2$}\cr
x+2, &\text{ if $b-1 \leq x$ }\cr
\end{cases}
$$
where $a=\min(X,Y)$ and $b=\max(X,Y)$.

\begin{proposition}
Let $1\leq X, Y \leq 2n$ and $X\not= Y$. Then 
\begin{equation}\label{eq.XY}
 f_n[\smatrix{X\cr Y\cr}] =
 \sum_{x,y} f_{n-1}[\smatrix{x\cr y\cr}]  
\end{equation}
where the sum ranges over all ordered pairs $(x,y)$ of $[2n-2]\times [2n-2]$
such that the reduction of
$[\smatrix{\Omega^{-1}(x; X, Y) & X \cr \Omega^{-1}(y; X, Y) & Y\cr }]$ is a piece in $\setP$.
\end{proposition}

Formulas (\ref{eq.sum}) and (\ref{eq.XY}) are 
used for computing the first values of the number of standard puzzles $(f_n)$.

\section{Selected examples}\label{sec.example} 
In Section \ref{sec.dict} we display the full list of all standard puzzle sequences
including the outputs from OEIS if any. Each item in the dictionary may contain
the following fields:
\vspace{\smallskipamount}

- Codes and pieces: The support $\setP$ of the standard puzzles
\vspace{\smallskipamount}

- See also: List of other supports $\setP'$ such  $|\setP^n|= |\setP'^n|$
\vspace{\smallskipamount}

- Seq: The sequence $(|\setP^n|)$ for $n=2,3,4,\ldots$
\vspace{\smallskipamount}

- Var: This is a variant of the sequence of $(|\setP^n|)$.
Sometimes the sequence is not in OEIS but
a slight modification of that sequence is in OEIS.
\vspace{\smallskipamount}

- OEIS: The OEIS output contains three parts:
the OEIS code, the number of results found (written in square brackets [\ ]), 
and the description of the first result found in OEIS.

\vspace{\medskipamount}

In the present paper we will not reproduce all the formulas
observed in the same way as formulas (\ref{eq.BC}-\ref{eq.EJRV})
	further derived.
In fact, we have 114 sequences identified in OEIS
up to order $|\setP|=4$. It corresponds to 309 supports $\setP$,
that is, we have 309 formulas to prove!
\vspace{\medskipamount}

For example, we find the following item in the dictionary.
\goodbreak
\vspace{\medskipamount}

\hrule
\nobreak
\vspace{\smallskipamount}
\nobreak
{
\parindent=0pt
\goodbreak{\bf  BC }\quad $\{\piece 1243, \piece 1324\}$
\quad See also 
{\bf BD}
\vspace{\smallskipamount}

{Seq= 2, 5, 14, 42, 132, 429, 1430, 4862, 16796, 58786, 208012 }

OEIS:  \seqnum{A000108} [5] Catalan numbers: C(n) = binomial(2n,n)/(n+1) = (2n)!/(n!(n+1)!). Also called Segner numbers.
\vspace{\smallskipamount}

}
\hrule
\vspace{\bigskipamount}
From the above item $BC$ in the dictionary we extract the following results.
\begin{theorem}
We have
\begin{eqnarray}
|BC^n|&=&|\bigl\{ \piece 1243, \piece 1324 \bigr\}^{n}| = C_n; \label{eq.BC}\\
|BD^n|&=&|\bigl\{ \piece 1243, \piece 1342 \bigr\}^{n}| = C_n,\label{eq.BD} 
\end{eqnarray}
where $C_n=\frac{1}{n+1} \binom{2n}{n}$ is the Catalan numbers 
\cite{Wiki_Catalan}.
\end{theorem}
\begin{proof}
Equation (\ref{eq.BD}) is well-known since the puzzles with support 
$BD=\bigl\{ \piece 1243, \piece 1342 \bigr\}$ are just the 
standard Young tableaux, which can be counted by the famous hook 
length formula \cite{Hook}. In fact, each puzzle with support $BD$
is just a labelling of the vertices of the following diagram 
such that the labels are increasing in the sense of the arrows \cite[p. 227]{Stanley}.
%
\long\def\pythonbegin#1\pythonend{}
\pythonbegin
beginfig(10, "1.8mm");  
setLegoUnit([6,6])
#u_showgrid([0,0], [8,8])  # show grid
r=0.12  # dot radius
def myarc(pa, pb):
	arc(pa, pb, marge=2*r, h=0, f=0.6, flen=0.01)

[circle([x,0], r) for x in range(5)] 
[circle([x,1], r) for x in range(5)] 
[myarc([x,1], [x+1,1]) for x in range(4)] 
[myarc([x,0], [x+1,0]) for x in range(4)] 
[myarc([x,0], [x,1]) for x in range(5)] 

endfig();
\pythonend
\begin{figure}[ht]
\begin{center}
\includegraphics{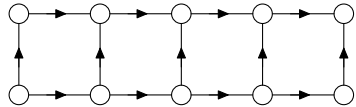}
\end{center}
\vspace{-12pt}
\caption{Puzzles having the support $BD$}
\label{fig.BD}
\end{figure}

However, identity (\ref{eq.BC}) cannot be directly found in the 
references. A further bijection is to be constructed to 
map a puzzle 
from $BC^n$ onto a puzzle from $BD^n$. 
Each puzzle with support $BC$ contains no piece $C$, or only one piece $C$
at the rightmost position.
Hence, the puzzles with support $BC$ can be characterized by the following 
two diagrams:
%
\long\def\pythonbegin#1\pythonend{}
\pythonbegin
beginfig(11, "1.8mm");  
setLegoUnit([6,6])
#u_showgrid([0,0], [8,8])  # show grid
r=0.12  # dot radius
def myarc(pa, pb):
	arc(pa, pb, marge=2*r, h=0, f=0.6, flen=0.01) 
def myarc2(pa, pb):
	arc(pa, pb, marge=2*r, h=0, f=0.4, flen=0.01)

[circle([x,0], r) for x in range(5)] 
[circle([x,1], r) for x in range(5)] 
[myarc([x,1], [x+1,1]) for x in range(4)] 
[myarc([x,0], [x+1,0]) for x in range(4)] 
[myarc([x+1,0], [x,1]) for x in range(4)] 

setOffset([40,0])

[circle([x,0], r) for x in range(5)] 
[circle([x,1], r) for x in range(5)] 
[myarc([x,1], [x+1,1]) for x in range(3)] 
[myarc([x,0], [x+1,0]) for x in range(3)] 
[myarc([x+1,0], [x,1]) for x in range(3)] 
myarc2([3,0],[4,1])
myarc2([4,0],[3,1])
myarc([4,1],[4,0])

endfig();
\pythonend
\begin{figure}[H]
\begin{center}
\includegraphics{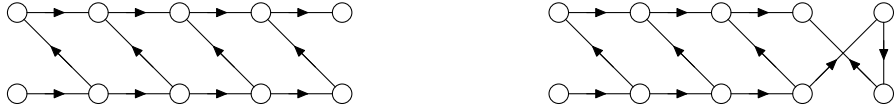}
\end{center}
\vspace{-12pt}
\caption{Puzzles having the support $BC$}
\label{fig.BC}
\end{figure}

In Fig. \ref{fig.BD} and \ref{fig.BC} the labels of certain vertices 
can only have fixed values, in particular maximal or minimal values. 
Removing those vertices yields the simplified diagrams Figures~\ref{fig.BD.reduction} and~\ref{fig.BC.reduction}.
The standard labellings of the diagram in Fig.~\ref{fig.BD.reduction} 
such that $a<b$ (resp. $a>b$) are in bijection with the standard labellings 
of the left diagram (resp. right diagram)
in Fig.~\ref{fig.BC.reduction}.
This gives a bijection between $BC^n$ and $BD^n$.
%
%
\long\def\pythonbegin#1\pythonend{}
\pythonbegin
beginfig(12, "1.8mm");  
setLegoUnit([6,6])
#u_showgrid([0,0], [8,8])  # show grid
r=0.12  # dot radius
def myarc(pa, pb):
	arc(pa, pb, marge=2*r, h=0, f=0.6, flen=0.01)

[circle([x+1,0], r) for x in range(4)] 
[circle([x,1], r) for x in range(4)] 
[myarc([x,1], [x+1,1]) for x in range(3)] 
[myarc([x+1,0], [x+2,0]) for x in range(3)] 
[myarc([x+1,0], [x+1,1]) for x in range(3)] 
label([3,1], "{$\scriptstyle a$}")
label([4,0], "{$\scriptstyle b$}")

endfig();
\pythonend
\begin{figure}[ht]
\begin{center}
\includegraphics{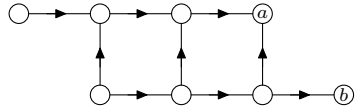}
\end{center}
\vspace{-12pt}
\caption{Puzzles having the support $BD$ (after simplification)}
\label{fig.BD.reduction}
\end{figure}
%
%
\long\def\pythonbegin#1\pythonend{}
\pythonbegin
beginfig(13, "1.8mm");  
setLegoUnit([6,6])
#u_showgrid([0,0], [8,8])  # show grid
r=0.12  # dot radius
def myarc(pa, pb):
	arc(pa, pb, marge=2*r, h=0, f=0.6, flen=0.01) 
def myarc2(pa, pb):
	arc(pa, pb, marge=2*r, h=0, f=0.4, flen=0.01) 

[circle([x+2,0], r) for x in range(3)] 
[circle([x,1], r) for x in range(3)] 
[myarc([x,1], [x+1,1]) for x in range(2)] 
[myarc([x+2,0], [x+3,0]) for x in range(2)] 
[myarc([x+2,0], [x+1,1]) for x in range(2)] 

setOffset([40,0])

[circle([x+2,0], r) for x in range(3)] 
[circle([x,1], r) for x in range(3)] 
[myarc([x,1], [x+1,1]) for x in range(2)] 
[myarc([x+2,0], [x+3,0]) for x in range(1)] 
[myarc([x+2,0], [x+1,1]) for x in range(2)] 
circle([4,1],r)
myarc([3,0],[4,1])
myarc([4,1],[4,0])

endfig();
\pythonend
\begin{figure}[ht]
\begin{center}
\includegraphics{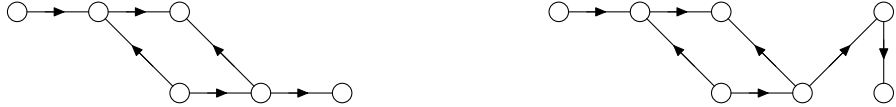}
\end{center}
\vspace{-12pt}
\caption{Puzzles having the support $BC$ (after simplification)}
\label{fig.BC.reduction}
\end{figure}
\end{proof}

Recall that the classical tangent numbers $T_{2n+1}$,
secant numbers $E_{2n}$ and the unsigned Genocchi numbers $G_{2n}$ are defined by the following generating functions \cite{Wiki_Tangent}.
\begin{eqnarray*}
 \tan u&=&\sum_{n\ge 0}\frac{u^{2n+1}}{ (2n+1)!}T_{2n+1}  \\
 \sec u&=&\sum_{n\ge 0} \frac{u^{2n}}{ (2n)!}E_{2n} \\
\frac{2u }{ {e^u+1}}&=&u+ \sum_{n \geq 1}(-1)^n  \frac{u^{2n}  }{ (2n)!}G_{2n}
\end{eqnarray*}
It is well-known that $T_{2n+1}$ and $E_{2n}$ count the numbers of
alternating permutations of length $2n+1$ and $2n$ respectively \cite{Wiki_Tangent}.
The unsigned Genocchi numbers $G_{2n}=n T_{2n-1}/2^{2n-2}$ count
the numbers of surjective staircases \cite{Dumont}. 
The support $\setP=BCEG$ involves those numbers.  
\newpage 
\vspace{\medskipamount}

\hrule
\vspace{\smallskipamount}

{\parindent=0pt
\goodbreak{\bf  BCEG } \ $\{\piece 1243, \piece 1324, \piece 1423, \piece 2134\}$
\vspace{\smallskipamount}
See also 
{\bf BEGJ}, {\bf BEGY}, {\bf BEJV}, {\bf BEVY}, {\bf BJTV}, {\bf EJRV}
\vspace{\smallskipamount}

{Seq= 4, 24, 272, 4960, 132672, 4893056, 237969664, 14756156928 }

{Var= 1, 3, 17, 155, 2073, 38227, 929569, 28820619, 1109652905 }

OEIS:  \seqnum{A110501} [2] Unsigned Genocchi numbers (of first kind) of even index.
\vspace{\smallskipamount}

}
\hrule

\vspace{\medskipamount}

From the above item $BCEG$ in the dictionary we read the following seven identities as shown in Theorem \ref{th.BJTV} and Conjecture \ref{conj.BCEG}.
\begin{theorem}\label{th.BJTV}
We have
\begin{eqnarray}
|BEGJ^n|  &=& |\{\piece 1243, \piece 1423, \piece 2134, \piece 2314\}^n|
  = nT_{2n-1}/2^{n-2}\,; \label{eq.BEGJ}\\
|BEGY^n|  &=& |\{\piece 1243, \piece 1423, \piece 2134, \piece 4312\}^n|
  = nT_{2n-1}/2^{n-2}\,; \label{eq.BEGY}\\
|BEJV^n|  &=& |\{\piece 1243, \piece 1423, \piece 2314, \piece 4132\}^n|
  = nT_{2n-1}/2^{n-2}\,; \label{eq.BEJV}\\
|BEVY^n|  &=& |\{\piece 1243, \piece 1423, \piece 4132, \piece 4312\}^n|
  = nT_{2n-1}/2^{n-2}\,; \label{eq.BEVY}\\
|BJTV^n|  &=& |\{\piece 1243, \piece 2314, \piece 3421, \piece 4132\}^n|
  = nT_{2n-1}/2^{n-2}\,; \label{eq.BJTV}\\
|EJRV^n|  &=& |\{\piece 1423, \piece 2314, \piece 3241, \piece 4132\}^n|
  = nT_{2n-1}/2^{n-2}\,. \label{eq.EJRV}
\end{eqnarray}
\end{theorem}
\begin{proof}
The enumeration of $|BJTV^n|$ is derived in \cite{GZ, FHa} and
identity (\ref{eq.BJTV}) is proven to be true. 
Notice that no easy direct proof of (\ref{eq.BJTV}) is known.
We prove that all the six left-hand sides of equations (\ref{eq.BEGJ}-\ref{eq.EJRV}) are equal to
\begin{equation}\label{eq.BGTY}
|BGTY^n|  = |\{\piece 1243, \piece 2134, \piece 3421, \piece 4312\}^n|\,.
\end{equation}

Let $S$ be a subset of $\{1,2,\ldots, n\}$. The {\it flip map} 
$\phi_S: \alpha \mapsto \beta$ is a transformation which maps
a puzzle 
$\alpha=\bigl[\smatrix{x_1 x_2 \cdots x_n\cr y_1 y_2 \cdots y_n}\bigr]$ 
onto
$\beta=\bigl[\smatrix{a_1 a_2 \cdots a_n\cr b_1 b_2 \cdots b_n}\bigr]$ 
such that
$a_i=x_i, b_i=y_i$ for $i\not\in S$ and
$a_i=y_i, b_i=x_i$ for $i\in S$.
Notice that the flip map is an involution.
The following diagram shows the actions of the 
flip maps onto some pieces:
$$
\begin{array}{ccc}
(B,G,T,Y) &\hfl{\phi_{\{1,2\}}}{}{14mm}& (T,Y,B,G) \\
\noalign{\vspace{6pt}}
\vfl{}{\phi_{\{2\}}}{5mm}  & & \\
\noalign{\vspace{6pt}}
(E,J,R,V) &\hfl{\phi_{\{1,2\}}}{}{14mm}& (R,V,E,J) 
\end{array}
$$
The {\it important facts} are that the flip map
$\phi_{\{1,2\}}$ does not change the two supports
$\{B,G,T,Y\}$ and $\{E,J,R,V\}$; moreover, those two supports
do not contain any common piece. We can prove that
all the $2^4=16$ supports $\setP$ derived from the Cartesian product
$\Gamma:=\{B,E\}\times \{G,J\} \times \{T,R\} \times \{Y,V\}$
yield the same enumeration sequences $(|\setP^n|)$. For example, we will
explain why
\begin{equation}\label{eq.tan.proof}
|BGTY^n| = |BJRY^n|\,.
\end{equation}
Notice that 
\begin{equation}
|BGTY^n| = |BG^n| + |TY^n| = 2 |BG^n|\,.
\end{equation}
Let $\alpha$ be a puzzle in $BGTY^n$. We construct 
a puzzle $\beta$ in $BJRY^n$ in a unique manner by applying 
an adequate sequence of flip maps.

Start at $\gamma:=\alpha$; from left to right look for the first piece 
in $\gamma$
that is not in $\{B,J,R,Y\}$, that is, the first piece equal to $G$ or $T$.
Let $(i,i+1)$ be the position of that piece. 
Apply the flip map 
$\phi_{\{i+1,i+2,\ldots, n-1\}}$ to $\gamma$, to obtain a new puzzle. 
By convention, let
$\gamma:=\phi_{\{i+1,i+2,\ldots, n-1\}}(\gamma)$. 
Hence, the piece at position $(i,i+1)$ becomes $J$ or $R$,
and all pieces on the right of that position are still in $\{B,G,T,Y\}$,
by the important facts mentioned above. Repeat this processus
until no more $G$ or $T$ are in $\gamma$ to get the puzzle $\beta:=\gamma$
in $BJRY^n$. 

For example, take 
$\alpha= \bigl[\smatrix{11 & 13 & 10 & 6 & 12 & 8 & 14\cr
7 & 9 & 2 & 1 & 4 & 3 & 5\cr}\bigr] \in BGTY^n$. 
The calculation in Table \ref{table.bij} shows
that
$\beta= \bigl[\smatrix{
11 & 13& 2& 1 &12 & 3 & 14\cr
7 &  9 & 10 & 6 & 4 & 8 & 5\cr
}\bigr] \in BJRY^n$. Note that the number of a piece, for example $B$,
common to two different supports,
is not preserved.
\begin{table}
{
\def\cc#1{{\rlap{\, \tiny #1}}}
$$
\begin{array}{ccccc}
&\text{puzzle} & \text{first $G/T$} & \text{position} & \text{flip} \cr
\noalign{\vspace{6pt}}
\rightarrow & \begin{bmatrix}
11 & 13& 10& 6 &12 & 8 & 14\cr
\cc B &\cc G &\cc G &\cc B & \cc G &\cc B &\cr
7 &  9 & 2 & 1 & 4 & 3 & 5\cr
\end{bmatrix} 
&  G& (2,3) & S=\{3,4,5,6,7\} \cr
\noalign{\vspace{6pt}}
\rightarrow & \begin{bmatrix}
11 & 13& 2& 1 &4 & 3 & 5\cr
\cc B &\cc J &\cc Y &\cc T & \cc Y &\cc T &\cr
7 &  9 & 10 & 6 & 12 & 8 & 14\cr
\end{bmatrix} 
&  T& (4,5) & S=\{5,6,7\} \cr
\noalign{\vspace{6pt}}
\rightarrow & \begin{bmatrix}
11 & 13& 2& 1 &12 & 8 & 14\cr
\cc B &\cc J &\cc Y &\cc R & \cc G &\cc B &\cr
7 &  9 & 10 & 6 & 4 & 3 & 5\cr
\end{bmatrix} 
&  G& (5,6) & S=\{6,7\} \cr
\noalign{\vspace{6pt}}
\rightarrow & \begin{bmatrix}
11 & 13& 2& 1 &12 & 3 & 5\cr
\cc B &\cc J &\cc Y &\cc R & \cc J &\cc T &\cr
7 &  9 & 10 & 6 & 4 & 8 & 14\cr
\end{bmatrix} 
&  T& (6,7) & S=\{7\} \cr
\noalign{\vspace{6pt}}
\rightarrow & \begin{bmatrix}
11 & 13& 2& 1 &12 & 3 & 14\cr
\cc B &\cc J &\cc Y &\cc R & \cc J &\cc R &\cr
7 &  9 & 10 & 6 & 4 & 8 & 5\cr
\end{bmatrix} 
&  &  &  \cr
\end{array}
$$
}
\label{table.bij}
\caption{Bijection from $BGTY^n$ onto $BJRY^n$}
\end{table}

This processus is reversible thanks to the important facts mentioned above.
Hence, identity (\ref{eq.tan.proof}) is proved.
Now, since $BJTV, EJRV \in \Gamma$, identities (\ref{eq.BJTV}) and
(\ref{eq.EJRV}) are proved. In the same manner, 
$BGRV, EGTY, BJRV, BJRY \in \Gamma$  
and $BGRV\equiv BEGJ,\, EGTY\equiv BEGY,\, BJRV\equiv BEJV,\, EJRY\equiv BEVY$.
This achiveves the proof of (\ref{eq.BEGJ}-\ref{eq.BEVY}).
\end{proof}

The next identity cannot be proved in the same manner, since the flip maps
are not enough to produce the bijection. Further operations are to be
constructed.
\begin{conjecture}\label{conj.BCEG}
We have
\begin{eqnarray}
|BCEG^n|  &=& |\{\piece 1243, \piece 1324, \piece 1423, \piece 2134\}^n| 
  = nT_{2n-1}/2^{n-2}\,. \label{eq.BCEG}
\end{eqnarray}
\end{conjecture}
\vspace{\medskipamount}
Other supports $\setP$ giving classical numbers are listed below without proof.
\vspace{\medskipamount}

{\parindent=0pt

{\bf ACX}: Fibonacci numbers;

{\bf AB, ABCD}: double factorial numbers;


{\bf CDW}: Koch snowflake, number of angles after n iterations;

{\bf ACET}: The number of branching configurations of RNA with n or fewer hairpins;

{\bf ACMT}: super-Catalan numbers or little Schroeder numbers;

{\bf ADHN}: Number of permutations of length 2n-1 with no local maxima or minima in even positions;

{\bf AELM}: Number of permutations of length n which avoid the patterns 231, 12534;

{\bf BDFK}: Number of Dyck paths of semilength n having no DUDU's starting at level 1.

}

\section{Where are the secant numbers ?}  
Since the tangent number sequence is a puzzle sequence with support $BCEG$ 
(using a slight modification, see (4.3)), a natural question rises:
what about the
secant numbers $(E_{2n})_{n\geq 2}$:
\begin{equation}\label{eq.secant}
(5, 61, 1385, 50521, 2702765, 199360981, 19391512145,\ldots)\ ?
\end{equation}

The initial motivation was to define a support 
that would generate a
puzzle sequence for 
the secant numbers. We developed a computer programme
to generate the puzzle sequences up to order $|\setP|=4$. Unfortunately,
no secant number sequence could be found.
\vspace{\medskipamount}

Notice that the secant numbers appeared in the puzzle sequence
for the support $BJTV$, but a non-trivial modification is required.
By (\ref{eq.BJTV}) we have
\begin{equation}\label{eq.BJTV.XY}
|BJTV^n|  =
\sum_{X,Y}|BJTV^n[\smatrix{X\cr Y\cr}]|  =
\frac{n}{ 2^{n-2}} \, T_{2n-1}, 
\end{equation}
where 
$BJTV^n[\smatrix{X\cr Y\cr}]$ is the subset of puzzles in $NJTV^n$
such that the rightmost colunm is $[\smatrix{X\cr Y\cr}]$.
In \cite{FHb} we obtained the following identity by using the signed doubloon
model:
\begin{equation}\label{eq.BJTV.Q}
\sum_{X,Y}|BJTV^n[\smatrix{X\cr Y\cr}]| \times Q_n(X,Y)  =
2^{-n}E_{2n}, 
\end{equation}
where 
$$
Q_n(X,Y)=
\begin{cases}
\sum_{k=X}^{Y-1} \binom{2n}{ k} &\text{ if $X<Y$}\cr
\noalign{\vspace{\smallskipamount}
}
0 &\text{ if $X>Y$ }\cr
\end{cases}
$$

However, the secant number sequence (\ref{eq.secant}) does not directly appear in the 
dictionary without the coefficient $Q_n(X,Y)$.
It is possible that order $|\setP|\leq 4$ is not large enough. 
Why do not choose
a bigger suport, for example $\setP=CEHJLPRVX$. The puzzle sequence is
\begin{equation}\label{eq.secant.approx}
(9, 111, 2505, 91961, 4913789, 364074545, 35418898477, 
\ldots)
\end{equation}
We don't obtain the secant sequence (\ref{eq.secant}). It is amusing that
the quotient of (\ref{eq.secant.approx}) by the secant sequence (\ref{eq.secant}) is approximately $1.8$:
\begin{equation*}
(1.800, 1.820, 1.809, 1.820, 1.818, 1.826, 1.827, \ldots)
\end{equation*}
What does it mean ?

%
%
%
%
%
%
%

\section{Dictionary of the standard puzzle sequences}\label{sec.dict}

{\parindent=0pt
\vspace{\medskipamount}

\goodbreak
\centerline{\bf Sequence identified for $|\setP|=1$}
\vspace{\medskipamount}

\goodbreak{\bf  A }\quad $\{\piece 1234\}$
\quad See also 
{\bf D}
\vspace{\smallskipamount}

{Seq= 1, 1, 1, 1, 1, 1, 1, 1, 1, 1, 1 }

OEIS:  \seqnum{A000012} [1450] The simplest sequence of positive numbers: the all 1's sequence.
\vspace{\medskipamount}

\goodbreak{\bf  B }\quad $\{\piece 1243\}$
\vspace{\smallskipamount}

{Seq= 1, 2, 5, 14, 42, 132, 429, 1430, 4862, 16796, 58786 }

OEIS:  \seqnum{A000108} [5] Catalan numbers: C(n) = binomial(2n,n)/(n+1) = (2n)!/(n!(n+1)!). Also called Segner numbers.
\vspace{\medskipamount}

\goodbreak{\bf  C }\quad $\{\piece 1324\}$
\quad See also 
{\bf E}, {\bf F}
\vspace{\smallskipamount}

{Seq= 1, 0, 0, 0, 0, 0, 0, 0, 0, 0, 0 }

OEIS:  \seqnum{A121373} [1470] Expansion of f(q) = f(q, -q\^{}2) in powers of q where f(q,r) is the Ramanujan two variable theta function.
\vspace{\medskipamount}

\goodbreak
\centerline{\bf Sequence identified for $|\setP|=2$}
\vspace{\medskipamount}

\goodbreak{\bf  AB }\quad $\{\piece 1234, \piece 1243\}$
\vspace{\smallskipamount}

{Seq= 2, 8, 48, 384, 3840, 46080, 645120, 10321920, 185794560, 3715891200 }

OEIS:  \seqnum{A000165} [1] Double factorial numbers: (2n)!! = 2\^{}n*n!.
\vspace{\medskipamount}

\goodbreak{\bf  AC }\quad $\{\piece 1234, \piece 1324\}$
\quad See also 
{\bf AK}, {\bf AX}, {\bf CD}, {\bf CX}, {\bf DE}, {\bf DF}, {\bf EK}, {\bf FK}
\vspace{\smallskipamount}

{Seq= 2, 2, 2, 2, 2, 2, 2, 2, 2, 2, 2 }

OEIS:  \seqnum{A055642} [327] Number of digits in decimal expansion of n.
\vspace{\medskipamount}

\goodbreak{\bf  AD }\quad $\{\piece 1234, \piece 1342\}$
\vspace{\smallskipamount}

{Seq= 2, 6, 23, 106, 567, 3434, 23137, 171174, 1376525, 11934581, 110817423 }

OEIS:  \seqnum{A125273} [1] Eigensequence of triangle A085478: a(n) = Sum\_\{k=0..n-1\} A085478(n-1,k)*a(k) for n$>$0 with a(0)=1.
\vspace{\medskipamount}

\goodbreak{\bf  AE }\quad $\{\piece 1234, \piece 1423\}$
\vspace{\smallskipamount}

{Seq= 2, 3, 4, 5, 6, 7, 8, 9, 10, 11, 12 }

OEIS:  \seqnum{A000027} [618] The natural numbers. Also called the whole numbers, the counting numbers or the positive integers.
\vspace{\medskipamount}

\goodbreak{\bf  AF }\quad $\{\piece 1234, \piece 1432\}$
\vspace{\smallskipamount}

{Seq= 2, 4, 7, 11, 16, 22, 29, 37, 46, 56, 67 }

OEIS:  \seqnum{A000124} [6] Central polygonal numbers (the Lazy Caterer's sequence): n(n+1)/2 + 1; or, maximal number of pieces formed when slicing a pa...
\vspace{\medskipamount}

\goodbreak{\bf  AL }\quad $\{\piece 1234, \piece 2413\}$
\vspace{\smallskipamount}

{Seq= 2, 5, 10, 17, 26, 37, 50, 65, 82, 101, 122 }

OEIS:  \seqnum{A002522} [2] n\^{}2 + 1.
\vspace{\medskipamount}

\goodbreak{\bf  AP }\quad $\{\piece 1234, \piece 3142\}$
\quad See also 
{\bf DQ}
\vspace{\smallskipamount}

{Seq= 2, 7, 16, 29, 46, 67, 92, 121, 154, 191, 232 }

OEIS:  \seqnum{A130883} [1] 2n\^{}2-n+1.
\vspace{\medskipamount}

\goodbreak{\bf  AR }\quad $\{\piece 1234, \piece 3241\}$
\quad See also 
{\bf CM}, {\bf DJ}
\vspace{\smallskipamount}

{Seq= 2, 4, 6, 8, 10, 12, 14, 16, 18, 20, 22 }

OEIS:  \seqnum{A005843} [25] The even numbers: a(n) = 2n.
\vspace{\medskipamount}

\goodbreak{\bf  BC }\quad $\{\piece 1243, \piece 1324\}$
\quad See also 
{\bf BD}
\vspace{\smallskipamount}

{Seq= 2, 5, 14, 42, 132, 429, 1430, 4862, 16796, 58786, 208012 }

OEIS:  \seqnum{A000108} [5] Catalan numbers: C(n) = binomial(2n,n)/(n+1) = (2n)!/(n!(n+1)!). Also called Segner numbers.
\vspace{\medskipamount}

\goodbreak{\bf  BE }\quad $\{\piece 1243, \piece 1423\}$
\quad See also 
{\bf ER}
\vspace{\smallskipamount}

{Seq= 2, 4, 10, 28, 84, 264, 858, 2860, 9724, 33592, 117572 }

OEIS:  \seqnum{A068875} [2] Expansion of (1+x*C)*C, where C = (1-(1-4*x)\^{}(1/2))/(2*x) is g.f. for Catalan numbers, A000108.
\vspace{\medskipamount}

\goodbreak{\bf  BF }\quad $\{\piece 1243, \piece 1432\}$
\vspace{\smallskipamount}

{Seq= 2, 3, 7, 19, 56, 174, 561, 1859, 6292, 21658, 75582 }

OEIS:  \seqnum{A005807} [2] Sum of adjacent Catalan numbers.
\vspace{\medskipamount}

\goodbreak{\bf  BG }\quad $\{\piece 1243, \piece 2134\}$
\vspace{\smallskipamount}

{Seq= 2, 12, 136, 2480, 66336, 2446528, 118984832, 7378078464, 568142287360 }

OEIS:  \seqnum{A117513} [1] Number of ways of arranging 2n tokens in a row, with 2 copies of each token from 1 through n, such that between every pair o...
\vspace{\medskipamount}

\goodbreak{\bf  BJ }\quad $\{\piece 1243, \piece 2314\}$
\vspace{\smallskipamount}

{Seq= 2, 6, 22, 84, 324, 1254, 4862, 18876, 73372, 285532, 1112412 }

OEIS:  \seqnum{A121686} [1] Number of branches in all binary trees with n edges. A binary tree is a rooted tree in which each vertex has at most two chi...
\vspace{\medskipamount}

\goodbreak{\bf  BL }\quad $\{\piece 1243, \piece 2413\}$
\vspace{\smallskipamount}

{Seq= 2, 4, 12, 40, 140, 504, 1848, 6864, 25740, 97240, 369512 }

OEIS:  \seqnum{A028329} [1] Twice central binomial coefficients.
\vspace{\medskipamount}

\goodbreak
\centerline{\bf Sequence unknown for $|\setP|=2$}
\vspace{\medskipamount}

\goodbreak{\bf  AH }\quad $\{\piece 1234, \piece 2143\}$
\vspace{\smallskipamount}

{Seq= 2, 12, 132, 2372, 62304, 2261668, 108184432, 6600715188, 500046044352 }
\vspace{\medskipamount}

\goodbreak{\bf  BN }\quad $\{\piece 1243, \piece 3124\}$
\vspace{\smallskipamount}

{Seq= 2, 9, 74, 974, 18831, 502459, 17671764, 792391014, 44129928926, 2987912108763 }
\vspace{\medskipamount}

\goodbreak{\bf  BQ }\quad $\{\piece 1243, \piece 3214\}$
\vspace{\smallskipamount}

{Seq= 2, 5, 21, 96, 440, 1989, 8855, 38896, 168948, 727090, 3105322 }
\vspace{\medskipamount}

\goodbreak{\bf  CK }\quad $\{\piece 1324, \piece 2341\}$
\vspace{\smallskipamount}

{Seq= 2, 4, 8, 26, 66, 276, 816, 4050, 13410, 75780, 274680 }
\vspace{\medskipamount}

\goodbreak{\bf  CP }\quad $\{\piece 1324, \piece 3142\}$
\vspace{\smallskipamount}

{Seq= 2, 10, 82, 1162, 23026, 657148, 23719394, 1137763610, 65032729314 }
\vspace{\medskipamount}

\goodbreak{\bf  CR }\quad $\{\piece 1324, \piece 3241\}$
\vspace{\smallskipamount}

{Seq= 2, 5, 18, 91, 563, 4299, 37686, 384543, 4357567, 55614775, 772479331 }
\vspace{\medskipamount}

\goodbreak{\bf  DN }\quad $\{\piece 1342, \piece 3124\}$
\vspace{\smallskipamount}

{Seq= 2, 14, 168, 3352, 96816, 3875904, 204185344, 13726330128, 1145508631264 }
\vspace{\medskipamount}

\goodbreak{\bf  EU }\quad $\{\piece 1423, \piece 4123\}$
\vspace{\smallskipamount}

{Seq= 2, 6, 26, 150, 1230, 10038, 125490, 1292166, 22184550, 271843110, 6022023210 }
\vspace{\medskipamount}

\goodbreak{\bf  EV }\quad $\{\piece 1423, \piece 4132\}$
\vspace{\smallskipamount}

{Seq= 2, 8, 58, 712, 12564, 310256, 10025978, 415159208, 21288518044, 1329526717840 }
\vspace{\medskipamount}

\goodbreak{\bf  FU }\quad $\{\piece 1432, \piece 4123\}$
\vspace{\smallskipamount}

{Seq= 2, 12, 106, 1680, 37434, 1171968, 48008850, 2516016384, 163509808050 }
\vspace{\medskipamount}

\goodbreak
\centerline{\bf Sequence identified for $|\setP|=3$}
\vspace{\medskipamount}

\goodbreak{\bf  ABC }\quad $\{\piece 1234, \piece 1243, \piece 1324\}$
\quad See also 
{\bf ABE}, {\bf ABF}
\vspace{\smallskipamount}

{Seq= 3, 12, 72, 576, 5760, 69120, 967680, 15482880, 278691840, 5573836800 }

OEIS:  \seqnum{A052676} [1] A simple regular expression in a labeled universe.
\vspace{\medskipamount}

\goodbreak{\bf  ABD }\quad $\{\piece 1234, \piece 1243, \piece 1342\}$
\vspace{\smallskipamount}

{Seq= 3, 15, 105, 945, 10395, 135135, 2027025, 34459425, 654729075, 13749310575 }

OEIS:  \seqnum{A001147} [1] Double factorial numbers: (2n-1)!! = 1.3.5....(2n-1).
\vspace{\medskipamount}

\goodbreak{\bf  ABK }\quad $\{\piece 1234, \piece 1243, \piece 2341\}$
\vspace{\smallskipamount}

{Seq= 3, 10, 56, 432, 4224, 49920, 691200, 10967040, 196116480, 3901685760 }

{Var= 3, 10, 28, 72, 176, 416, 960, 2176, 4864, 10752, 23552 }

OEIS:  \seqnum{A128135} [1] Row sums of A128134.
\vspace{\medskipamount}

\goodbreak{\bf  ACE }\quad $\{\piece 1234, \piece 1324, \piece 1423\}$
\quad See also 
{\bf ACZ}, {\bf AKM}, {\bf AKS}, {\bf AMX}, {\bf DEM}, {\bf DFM}
\vspace{\smallskipamount}

{Seq= 3, 4, 5, 6, 7, 8, 9, 10, 11, 12, 13 }

OEIS:  \seqnum{A000027} [527] The natural numbers. Also called the whole numbers, the counting numbers or the positive integers.
\vspace{\medskipamount}

\goodbreak{\bf  ACF }\quad $\{\piece 1234, \piece 1324, \piece 1432\}$
\vspace{\smallskipamount}

{Seq= 3, 5, 8, 12, 17, 23, 30, 38, 47, 57, 68 }

OEIS:  \seqnum{A022856} [4] a(n) = n-2 + Sum of a(i+1)mod(a(i)) for i = 1 to n-2, for n $\geq$ 3.
\vspace{\medskipamount}

\goodbreak{\bf  ACL }\quad $\{\piece 1234, \piece 1324, \piece 2413\}$
\quad See also 
{\bf AMR}, {\bf CDM}
\vspace{\smallskipamount}

{Seq= 3, 6, 11, 18, 27, 38, 51, 66, 83, 102, 123 }

OEIS:  \seqnum{A059100} [4] n\^{}2+2.
\vspace{\medskipamount}

\goodbreak{\bf  ACM }\quad $\{\piece 1234, \piece 1324, \piece 2431\}$
\quad See also 
{\bf AFM}, {\bf AFZ}
\vspace{\smallskipamount}

{Seq= 3, 6, 14, 35, 90, 234, 611, 1598, 4182, 10947, 28658 }

OEIS:  \seqnum{A032908} [2] One of 4 3rd-order recurring sequences for which the first derived sequence and the Galois transformed sequence coincide.
\vspace{\medskipamount}

\goodbreak{\bf  ACS }\quad $\{\piece 1234, \piece 1324, \piece 3412\}$
\quad See also 
{\bf ATX}
\vspace{\smallskipamount}

{Seq= 3, 7, 21, 71, 253, 925, 3433, 12871, 48621, 184757, 705433 }

{Var= 0, 4, 18, 68, 250, 922, 3430, 12868, 48618, 184754, 705430 }

OEIS:  \seqnum{A115112} [1] Number of different ways to select n elements from two sets of n elements under the precondition of choosing at least one el...
\vspace{\medskipamount}

\goodbreak{\bf  ACT }\quad $\{\piece 1234, \piece 1324, \piece 3421\}$
\vspace{\smallskipamount}

{Seq= 3, 6, 17, 62, 259, 1162, 5441, 26234, 129283, 648142, 3294865 }

{Var= 2, 5, 16, 61, 258, 1161, 5440, 26233, 129282, 648141, 3294864 }

OEIS:  \seqnum{A104858} [1] Partial sums of the little Schroeder numbers (A001003).
\vspace{\medskipamount}

\goodbreak{\bf  ACX }\quad $\{\piece 1234, \piece 1324, \piece 4231\}$
\quad See also 
{\bf DEK}, {\bf DFK}
\vspace{\smallskipamount}

{Seq= 3, 5, 8, 13, 21, 34, 55, 89, 144, 233, 377 }

OEIS:  \seqnum{A000045} [15] Fibonacci numbers: F(n) = F(n-1) + F(n-2), F(0) = 0, F(1) = 1, F(2) = 1, ...
\vspace{\medskipamount}

\goodbreak{\bf  AEF }\quad $\{\piece 1234, \piece 1423, \piece 1432\}$
\quad See also 
{\bf AEQ}
\vspace{\smallskipamount}

{Seq= 3, 6, 10, 15, 21, 28, 36, 45, 55, 66, 78 }

OEIS:  \seqnum{A000217} [29] Triangular numbers: a(n) = C(n+1,2) = n(n+1)/2 = 0+1+2+...+n.
\vspace{\medskipamount}

\goodbreak{\bf  AEJ }\quad $\{\piece 1234, \piece 1423, \piece 2314\}$
\quad See also 
{\bf AKR}, {\bf AKV}, {\bf ARX}, {\bf CDJ}, {\bf CDL}, {\bf CEM}, {\bf CFM}, {\bf DEJ}, {\bf DFJ}
\vspace{\smallskipamount}

{Seq= 3, 5, 7, 9, 11, 13, 15, 17, 19, 21, 23 }

OEIS:  \seqnum{A005408} [34] The odd numbers: a(n) = 2n+1.
\vspace{\medskipamount}

\goodbreak{\bf  AEL }\quad $\{\piece 1234, \piece 1423, \piece 2413\}$
\quad See also 
{\bf AFQ}
\vspace{\smallskipamount}

{Seq= 3, 7, 13, 21, 31, 43, 57, 73, 91, 111, 133 }

OEIS:  \seqnum{A002061} [2] Central polygonal numbers: n\^{}2 - n + 1.
\vspace{\medskipamount}

\goodbreak{\bf  AEM }\quad $\{\piece 1234, \piece 1423, \piece 2431\}$
\vspace{\smallskipamount}

{Seq= 3, 5, 10, 23, 57, 146, 379, 989, 2586, 6767, 17713 }

{Var= 0, 2, 7, 20, 54, 143, 376, 986, 2583, 6764, 17710 }

OEIS:  \seqnum{A035508} [2] Fibonacci(2n+2)-1.
\vspace{\medskipamount}

\goodbreak{\bf  AES }\quad $\{\piece 1234, \piece 1423, \piece 3412\}$
\vspace{\smallskipamount}

{Seq= 3, 6, 17, 58, 212, 794, 3005, 11442, 43760, 167962, 646648 }

{Var= 1, 4, 15, 56, 210, 792, 3003, 11440, 43758, 167960, 646646 }

OEIS:  \seqnum{A001791} [2] Binomial coefficients C(2n,n-1).
\vspace{\medskipamount}

\goodbreak{\bf  AEW }\quad $\{\piece 1234, \piece 1423, \piece 4213\}$
\vspace{\smallskipamount}

{Seq= 3, 7, 22, 73, 237, 746, 2287, 6867, 20286, 59157, 170713 }

{Var= 1, 5, 20, 71, 235, 744, 2285, 6865, 20284, 59155, 170711 }

OEIS:  \seqnum{A054444} [1] Even indexed members of A001629(n), n $\geq$ 2, (Fibonacci convolution).
\vspace{\medskipamount}

\goodbreak{\bf  AEZ }\quad $\{\piece 1234, \piece 1423, \piece 4321\}$
\vspace{\smallskipamount}

{Seq= 3, 7, 19, 52, 140, 372, 981, 2577, 6757, 17702, 46358 }

{Var= 0, 4, 16, 49, 137, 369, 978, 2574, 6754, 17699, 46355 }

OEIS:  \seqnum{A114185} [1] F(2n)-n-1, where F(n)=Fibonacci number.
\vspace{\medskipamount}

\goodbreak{\bf  AFL }\quad $\{\piece 1234, \piece 1432, \piece 2413\}$
\vspace{\smallskipamount}

{Seq= 3, 8, 16, 27, 41, 58, 78, 101, 127, 156, 188 }

OEIS:  \seqnum{A104249} [1] (3*n\^{}2+n+2)/2.
\vspace{\medskipamount}

\goodbreak{\bf  AFS }\quad $\{\piece 1234, \piece 1432, \piece 3412\}$
\vspace{\smallskipamount}

{Seq= 3, 8, 31, 132, 564, 2382, 9951, 41228, 169768, 695862, 2842228 }

{Var= 1, 6, 29, 130, 562, 2380, 9949, 41226, 169766, 695860, 2842226 }

OEIS:  \seqnum{A008549} [1] Number of ways of choosing at most n-1 items from a set of size 2n+1.
\vspace{\medskipamount}

\goodbreak{\bf  AKP }\quad $\{\piece 1234, \piece 2341, \piece 3142\}$
\quad See also 
{\bf APX}, {\bf CDQ}, {\bf DEQ}, {\bf DFQ}
\vspace{\smallskipamount}

{Seq= 3, 8, 17, 30, 47, 68, 93, 122, 155, 192, 233 }

OEIS:  \seqnum{A033816} [1] 2n\^{}2 + 3n + 3.
\vspace{\medskipamount}

\goodbreak{\bf  AKT }\quad $\{\piece 1234, \piece 2341, \piece 3421\}$
\quad See also 
{\bf BFM}
\vspace{\smallskipamount}

{Seq= 3, 5, 10, 24, 66, 198, 627, 2057, 6919, 23715, 82501 }

OEIS:  \seqnum{A155587} [1] Expansion of (1+x*c(x))/(1-x), c(x) the g.f. of A000108.
\vspace{\medskipamount}

\goodbreak{\bf  AKU }\quad $\{\piece 1234, \piece 2341, \piece 4123\}$
\quad See also 
{\bf AKX}, {\bf CDE}, {\bf CDF}, {\bf DEF}
\vspace{\smallskipamount}

{Seq= 3, 3, 3, 3, 3, 3, 3, 3, 3, 3, 3 }

OEIS:  \seqnum{A010701} [202] Constant sequence.
\vspace{\medskipamount}

\goodbreak{\bf  AKW }\quad $\{\piece 1234, \piece 2341, \piece 4213\}$
\quad See also 
{\bf APS}, {\bf DFW}
\vspace{\smallskipamount}

{Seq= 3, 9, 33, 129, 513, 2049, 8193, 32769, 131073, 524289, 2097153 }

OEIS:  \seqnum{A084508} [2] Partial sums of A084509. Positions of ones in the first differences of A084506.
\vspace{\medskipamount}

\goodbreak{\bf  AMV }\quad $\{\piece 1234, \piece 2431, \piece 4132\}$
\quad See also 
{\bf ARS}
\vspace{\smallskipamount}

{Seq= 3, 8, 31, 126, 509, 2044, 8187, 32762, 131065, 524280, 2097143 }

{Var= 0, 5, 28, 123, 506, 2041, 8184, 32759, 131062, 524277, 2097140 }

OEIS:  \seqnum{A124133} [1] a(n)=(-1/2)*sum\_\{i1+i2+i3=2n\} ((2*n)!/(i1! i2! i3!))*B(i1) where B are the Bernoulli numbers.
\vspace{\medskipamount}

\goodbreak{\bf  APR }\quad $\{\piece 1234, \piece 3142, \piece 3241\}$
\quad See also 
{\bf CFW}, {\bf DJQ}
\vspace{\smallskipamount}

{Seq= 3, 10, 21, 36, 55, 78, 105, 136, 171, 210, 253 }

OEIS:  \seqnum{A014105} [1] Second hexagonal numbers: n(2n+1).
\vspace{\medskipamount}

\goodbreak{\bf  ART }\quad $\{\piece 1234, \piece 3241, \piece 3421\}$
\quad See also 
{\bf BET}
\vspace{\smallskipamount}

{Seq= 3, 8, 25, 84, 294, 1056, 3861, 14300, 53482, 201552, 764218 }

OEIS:  \seqnum{A038665} [1] Convolution of A007054 (super ballot numbers) with A000984 (central binomial coefficients).
\vspace{\medskipamount}

\goodbreak{\bf  ARV }\quad $\{\piece 1234, \piece 3241, \piece 4132\}$
\quad See also 
{\bf CEW}
\vspace{\smallskipamount}

{Seq= 3, 7, 11, 15, 19, 23, 27, 31, 35, 39, 43 }

OEIS:  \seqnum{A004767} [7] 4n+3.
\vspace{\medskipamount}

\goodbreak{\bf  BCD }\quad $\{\piece 1243, \piece 1324, \piece 1342\}$
\vspace{\smallskipamount}

{Seq= 3, 9, 28, 90, 297, 1001, 3432, 11934, 41990, 149226, 534888 }

OEIS:  \seqnum{A000245} [3] 3(2n)!/((n+2)!(n-1)!).
\vspace{\medskipamount}

\goodbreak{\bf  BCE }\quad $\{\piece 1243, \piece 1324, \piece 1423\}$
\quad See also 
{\bf BDF}, {\bf BEK}
\vspace{\smallskipamount}

{Seq= 3, 7, 19, 56, 174, 561, 1859, 6292, 21658, 75582, 266798 }

OEIS:  \seqnum{A005807} [3] Sum of adjacent Catalan numbers.
\vspace{\medskipamount}

\goodbreak{\bf  BDE }\quad $\{\piece 1243, \piece 1342, \piece 1423\}$
\vspace{\smallskipamount}

{Seq= 3, 8, 23, 70, 222, 726, 2431, 8294, 28730, 100776, 357238 }

OEIS:  \seqnum{A000782} [2] 2*Catalan(n)-Catalan(n-1).
\vspace{\medskipamount}

\goodbreak{\bf  BDL }\quad $\{\piece 1243, \piece 1342, \piece 2413\}$
\vspace{\smallskipamount}

{Seq= 3, 10, 35, 126, 462, 1716, 6435, 24310, 92378, 352716, 1352078 }

OEIS:  \seqnum{A001700} [3] C(2n+1, n+1): number of ways to put n+1 indistinguishable balls into n+1 distinguishable boxes = number of (n+1)-st degree m...
\vspace{\medskipamount}

\goodbreak{\bf  BEM }\quad $\{\piece 1243, \piece 1423, \piece 2431\}$
\vspace{\smallskipamount}

{Seq= 3, 6, 14, 37, 107, 329, 1055, 3486, 11780, 40510, 141286 }

OEIS:  \seqnum{A081293} [1] a(n) = A000108(n)+A014137(n).
\vspace{\medskipamount}

\goodbreak{\bf  BFL }\quad $\{\piece 1243, \piece 1432, \piece 2413\}$
\vspace{\smallskipamount}

{Seq= 3, 5, 14, 45, 154, 546, 1980, 7293, 27170, 102102, 386308 }

OEIS:  \seqnum{A078718} [1] Let f(i,j) = Sum(binom(2*i,k)*binom(2*j,i+j-k)*(-1)\^{}(i+j-k),k=0..2*i) (this is essentially the same as the triangle in...
\vspace{\medskipamount}

\goodbreak{\bf  BFT }\quad $\{\piece 1243, \piece 1432, \piece 3421\}$
\vspace{\smallskipamount}

{Seq= 3, 6, 15, 42, 126, 396, 1287, 4290, 14586, 50388, 176358 }

OEIS:  \seqnum{A120589} [1] Self-convolution of A120588, such that a(n) = 3*A120588(n) for n$\geq$2.
\vspace{\medskipamount}

\goodbreak{\bf  BLM }\quad $\{\piece 1243, \piece 2413, \piece 2431\}$
\vspace{\smallskipamount}

{Seq= 3, 6, 16, 51, 177, 639, 2355, 8790, 33100, 125478, 478194 }

{Var= 1, 4, 14, 49, 175, 637, 2353, 8788, 33098, 125476, 478192 }

OEIS:  \seqnum{A079309} [1] a(n) = C(1,1)+C(3,2)+C(5,3)+...+C(2n-1,n).
\vspace{\medskipamount}

\goodbreak{\bf  CDW }\quad $\{\piece 1324, \piece 1342, \piece 4213\}$
\quad See also 
{\bf DEW}
\vspace{\smallskipamount}

{Seq= 3, 6, 18, 66, 258, 1026, 4098, 16386, 65538, 262146, 1048578 }

OEIS:  \seqnum{A178789} [1] Koch snowflake: number of angles after n iterations.
\vspace{\medskipamount}

\goodbreak{\bf  DJM }\quad $\{\piece 1342, \piece 2314, \piece 2431\}$
\vspace{\smallskipamount}

{Seq= 3, 8, 21, 46, 87, 148, 233, 346, 491, 672, 893 }

OEIS:  \seqnum{A179903} [1] (1, 3, 5, 7, 9...) convolved with (1, 0, 3, 5, 7, 9,...)
\vspace{\medskipamount}

\goodbreak{\bf  EFK }\quad $\{\piece 1423, \piece 1432, \piece 2341\}$
\vspace{\smallskipamount}

{Seq= 3, 4, 6, 8, 12, 16, 24, 32, 48, 64, 96 }

OEIS:  \seqnum{A029744} [8] Numbers of the form 2\^{}n or 3*2\^{}n.
\vspace{\medskipamount}

\goodbreak
\centerline{\bf Sequence unknown for $|\setP|=3$ (partial)}
\vspace{\medskipamount}

\goodbreak{\bf  ABG }\quad $\{\piece 1234, \piece 1243, \piece 2134\}$
\vspace{\smallskipamount}

{Seq= 3, 23, 327, 7465, 249885, 11532671, 701867995, 54461600179, 5247921916235 }
\vspace{\medskipamount}

\goodbreak{\bf  ABH }\quad $\{\piece 1234, \piece 1243, \piece 2143\}$
\vspace{\smallskipamount}

{Seq= 3, 24, 345, 7920, 264873, 12190108, 739050425, 57109234080, 5479466654645 }
\vspace{\medskipamount}

\goodbreak{\bf  ABJ }\quad $\{\piece 1234, \piece 1243, \piece 2314\}$
\quad See also 
{\bf ABL}
\vspace{\smallskipamount}

{Seq= 3, 14, 96, 858, 9420, 122490, 1839600, 31325490, 596291220, 12546094050 }
\vspace{\medskipamount}

\goodbreak{\bf  ABN }\quad $\{\piece 1234, \piece 1243, \piece 3124\}$
\vspace{\smallskipamount}

{Seq= 3, 19, 217, 3985, 107547, 4001027, 196224625, 12270923649, 953000374835 }
\vspace{\medskipamount}

\goodbreak{\bf  ABP }\quad $\{\piece 1234, \piece 1243, \piece 3142\}$
\vspace{\smallskipamount}

{Seq= 3, 17, 116, 1048, 11712, 155520, 2388480, 41610240, 810270720, 17433722880 }
\vspace{\medskipamount}

\goodbreak{\bf  ABQ }\quad $\{\piece 1234, \piece 1243, \piece 3214\}$
\vspace{\smallskipamount}

{Seq= 3, 14, 104, 1020, 12264, 173580, 2818080, 51535260, 1047274200, 23400192060 }
\vspace{\medskipamount}

$\quad\vdots$

\goodbreak
\centerline{\bf Sequence identified for $|\setP|=4$}
\vspace{\medskipamount}

\goodbreak{\bf  ABCD }\quad $\{\piece 1234, \piece 1243, \piece 1324, \piece 1342\}$
\quad See also 
{\bf ABDE}, {\bf ABDF}
\vspace{\smallskipamount}

{Seq= 4, 20, 140, 1260, 13860, 180180, 2702700, 45945900, 872972100, 18332414100 }

{Var= 1, 5, 35, 315, 3465, 45045, 675675, 11486475, 218243025, 4583103525 }

OEIS:  \seqnum{A051577} [2] (2*n+3)!!/3, related to A001147 (odd double factorials).
\vspace{\medskipamount}

\goodbreak{\bf  ABCE }\quad $\{\piece 1234, \piece 1243, \piece 1324, \piece 1423\}$
\quad See also 
{\bf ABCF}, {\bf ABCM}, {\bf ABCT}, {\bf ABEF}, {\bf ABEZ}, {\bf ABRX}, {\bf ABRZ}, {\bf ABTX}, {\bf ACER}, {\bf ACRT}, {\bf AERX}, {\bf AERZ}, {\bf AETX}, {\bf BCEX}, {\bf BCRX}, {\bf BCTX}, {\bf CERX}
\vspace{\smallskipamount}

{Seq= 4, 16, 96, 768, 7680, 92160, 1290240, 20643840, 371589120, 7431782400 }

OEIS:  \seqnum{A032184} [2] "CIJ" (necklace, indistinct, labeled) transform of 1,3,5,7...
\vspace{\medskipamount}

\goodbreak{\bf  ABCJ }\quad $\{\piece 1234, \piece 1243, \piece 1324, \piece 2314\}$
\qquad See also 
{\bf ABCK}, {\bf ABCL}, {\bf ABDK}, {\bf ABEJ}, {\bf ABEL}, {\bf ABFJ}, {\bf ABFL}
\vspace{\smallskipamount}

{Seq= 4, 18, 120, 1050, 11340, 145530, 2162160, 36486450, 689188500, 14404039650 }

{Var= 2, 6, 30, 210, 1890, 20790, 270270, 4054050, 68918850, 1309458150 }

OEIS:  \seqnum{A097801} [2] (2*n)!/(n!*2\^{}(n-1)).
\vspace{\medskipamount}

\goodbreak{\bf  ABCZ }\quad $\{\piece 1234, \piece 1243, \piece 1324, \piece 4321\}$
\quad See also 
{\bf ABEM}, {\bf ABFM}, {\bf ABFZ}
\vspace{\smallskipamount}

{Seq= 4, 14, 80, 634, 6332, 75974, 1063624, 17017970, 306323444, 6126468862 }

{Var= 3, 13, 79, 633, 6331, 75973, 1063623, 17017969, 306323443, 6126468861 }

OEIS:  \seqnum{A010844} [1] a(n) = 2*n*a(n-1) + 1 with a(0)=1.
\vspace{\medskipamount}

\goodbreak{\bf  ABDJ }\quad $\{\piece 1234, \piece 1243, \piece 1342, \piece 2314\}$
\quad See also 
{\bf ABDL}
\vspace{\smallskipamount}

{Seq= 4, 24, 192, 1920, 23040, 322560, 5160960, 92897280, 1857945600, 40874803200 }

OEIS:  \seqnum{A002866} [3] a(0) = 1; for n$>$0, a(n) = 2\^{}(n-1)*n!.
\vspace{\medskipamount}

\goodbreak{\bf  ABDP }\quad $\{\piece 1234, \piece 1243, \piece 1342, \piece 3142\}$
\quad See also 
{\bf ABDU}, {\bf ABDV}
\vspace{\smallskipamount}

{Seq= 4, 25, 210, 2205, 27720, 405405, 6756750, 126351225, 2618916300, 59580345825 }

{Var= 1, 5, 35, 315, 3465, 45045, 675675, 11486475, 218243025, 4583103525 }

OEIS:  \seqnum{A051577} [2] (2*n+3)!!/3, related to A001147 (odd double factorials).
\vspace{\medskipamount}

\goodbreak{\bf  ABKM }\quad $\{\piece 1234, \piece 1243, \piece 2341, \piece 2431\}$
\quad See also 
{\bf ABKS}
\vspace{\smallskipamount}

{Seq= 4, 12, 60, 444, 4284, 50364, 695484, 11017404, 196811964, 3912703164 }

OEIS:  \seqnum{A004400} [1] 1 + Sum 2\^{}k k!, k = 1 . . n.
\vspace{\medskipamount}

\goodbreak{\bf  ABMX }\quad $\{\piece 1234, \piece 1243, \piece 2431, \piece 4231\}$
\vspace{\smallskipamount}

{Seq= 4, 13, 73, 577, 5761, 69121, 967681, 15482881, 278691841, 5573836801 }

{Var= 3, 12, 72, 576, 5760, 69120, 967680, 15482880, 278691840, 5573836800 }

OEIS:  \seqnum{A052676} [1] A simple regular expression in a labeled universe.
\vspace{\medskipamount}

\goodbreak{\bf  ACDF }\quad $\{\piece 1234, \piece 1324, \piece 1342, \piece 1432\}$
\quad See also 
{\bf ACDM}, {\bf ACFK}, {\bf ACKM}, {\bf ADFM}, {\bf ADFZ}, {\bf ADKX}, {\bf ADKZ}, {\bf ADMX}, {\bf AFKX}, {\bf AFKZ}, {\bf AFMX}, {\bf CDFX}, {\bf CDKX}, {\bf CDMX}, {\bf CFKX}
\vspace{\smallskipamount}

{Seq= 4, 12, 46, 212, 1134, 6868, 46274, 342348, 2753050, 23869162, 221634846 }

{Var= 2, 6, 23, 106, 567, 3434, 23137, 171174, 1376525, 11934581, 110817423 }

OEIS:  \seqnum{A125273} [1] Eigensequence of triangle A085478: a(n) = Sum\_\{k=0..n-1\} A085478(n-1,k)*a(k) for n$>$0 with a(0)=1.
\vspace{\medskipamount}

\goodbreak{\bf  ACEF }\quad $\{\piece 1234, \piece 1324, \piece 1423, \piece 1432\}$
\quad See also 
{\bf ACEQ}
\vspace{\smallskipamount}

{Seq= 4, 7, 11, 16, 22, 29, 37, 46, 56, 67, 79 }

OEIS:  \seqnum{A000124} [7] Central polygonal numbers (the Lazy Caterer's sequence): n(n+1)/2 + 1; or, maximal number of pieces formed when slicing a pa...
\vspace{\medskipamount}

\goodbreak{\bf  ACEJ }\quad $\{\piece 1234, \piece 1324, \piece 1423, \piece 2314\}$
\quad See also 
{\bf AKMX}, {\bf AKRU}, {\bf AKRX}, {\bf AKSU}, {\bf AKVX}, {\bf CDEJ}, {\bf CDEL}, {\bf CDFJ}, {\bf CDFL}, {\bf CEFM}, {\bf DEFJ}, {\bf DEFM}
\vspace{\smallskipamount}

{Seq= 4, 6, 8, 10, 12, 14, 16, 18, 20, 22, 24 }

OEIS:  \seqnum{A005843} [41] The even numbers: a(n) = 2n.
\vspace{\medskipamount}

\goodbreak{\bf  ACEL }\quad $\{\piece 1234, \piece 1324, \piece 1423, \piece 2413\}$
\quad See also 
{\bf ACFQ}, {\bf AKMR}, {\bf AMRX}, {\bf CDEM}, {\bf CDFM}
\vspace{\smallskipamount}

{Seq= 4, 8, 14, 22, 32, 44, 58, 74, 92, 112, 134 }

OEIS:  \seqnum{A014206} [2] n\^{}2+n+2.
\vspace{\medskipamount}

\goodbreak{\bf  ACEM }\quad $\{\piece 1234, \piece 1324, \piece 1423, \piece 2431\}$
\quad See also 
{\bf ACEZ}, {\bf AEFM}
\vspace{\smallskipamount}

{Seq= 4, 9, 22, 56, 145, 378, 988, 2585, 6766, 17712, 46369 }

OEIS:  \seqnum{A055588} [2] a(n)=3a(n-1)-a(n-2)-1; a(0)=1, a(1)=2.
\vspace{\medskipamount}

\goodbreak{\bf  ACES }\quad $\{\piece 1234, \piece 1324, \piece 1423, \piece 3412\}$
\quad See also 
{\bf AMTX}, {\bf BCFM}
\vspace{\smallskipamount}

{Seq= 4, 11, 36, 127, 463, 1717, 6436, 24311, 92379, 352717, 1352079 }

OEIS:  \seqnum{A112849} [1] Number of congruence classes (epimorphisms/vertex partitionings induced by graph endomorphisms) of undirected cycles of even...
\vspace{\medskipamount}

\goodbreak{\bf  ACET }\quad $\{\piece 1234, \piece 1324, \piece 1423, \piece 3421\}$
\vspace{\smallskipamount}

{Seq= 4, 10, 32, 122, 516, 2322, 10880, 52466, 258564, 1296282, 6589728 }

OEIS:  \seqnum{A176006} [1] The number of branching configurations of RNA (see Sankoff, 1985) with n or fewer hairpins.
\vspace{\medskipamount}

\goodbreak{\bf  ACEW }\quad $\{\piece 1234, \piece 1324, \piece 1423, \piece 4213\}$
\quad See also 
{\bf AEQW}
\vspace{\smallskipamount}

{Seq= 4, 11, 34, 106, 325, 978, 2896, 8463, 24466, 70102, 199369 }

{Var= 3, 10, 33, 105, 324, 977, 2895, 8462, 24465, 70101, 199368 }

OEIS:  \seqnum{A027989} [1] a(n) = self-convolution of row n of array T given by A027926.
\vspace{\medskipamount}

\goodbreak{\bf  ACFL }\quad $\{\piece 1234, \piece 1324, \piece 1432, \piece 2413\}$
\vspace{\smallskipamount}

{Seq= 4, 9, 17, 28, 42, 59, 79, 102, 128, 157, 189 }

{Var= 0, 5, 13, 24, 38, 55, 75, 98, 124, 153, 185 }

OEIS:  \seqnum{A140090} [1] n(3n+7)/2.
\vspace{\medskipamount}

\goodbreak{\bf  ACFM }\quad $\{\piece 1234, \piece 1324, \piece 1432, \piece 2431\}$
\quad See also 
{\bf AFQZ}
\vspace{\smallskipamount}

{Seq= 4, 10, 26, 68, 178, 466, 1220, 3194, 8362, 21892, 57314 }

OEIS:  \seqnum{A052995} [3] A simple regular expression.
\vspace{\medskipamount}

\goodbreak{\bf  ACLS }\quad $\{\piece 1234, \piece 1324, \piece 2413, \piece 3412\}$
\quad See also 
{\bf ARTX}, {\bf BCEM}, {\bf BDET}, {\bf BDFL}
\vspace{\smallskipamount}

{Seq= 4, 12, 40, 140, 504, 1848, 6864, 25740, 97240, 369512, 1410864 }

OEIS:  \seqnum{A100320} [2] A Catalan transform of (1+2x)/(1-2x).
\vspace{\medskipamount}

\goodbreak{\bf  ACMT }\quad $\{\piece 1234, \piece 1324, \piece 2431, \piece 3421\}$
\vspace{\smallskipamount}

{Seq= 4, 12, 46, 198, 904, 4280, 20794, 103050, 518860, 2646724, 13648870 }

{Var= 3, 11, 45, 197, 903, 4279, 20793, 103049, 518859, 2646723, 13648869 }

OEIS:  \seqnum{A001003} [2] Schroeder's second problem (generalized parentheses); also called super-Catalan numbers or little Schroeder numbers.
\vspace{\medskipamount}

\goodbreak{\bf  ACXZ }\quad $\{\piece 1234, \piece 1324, \piece 4231, \piece 4321\}$
\quad See also 
{\bf DEFK}, {\bf DEKM}, {\bf DFKM}
\vspace{\smallskipamount}

{Seq= 4, 8, 16, 32, 64, 128, 256, 512, 1024, 2048, 4096 }

OEIS:  \seqnum{A000079} [30] Powers of 2: a(n) = 2\^{}n.
\vspace{\medskipamount}

\goodbreak{\bf  ADHN }\quad $\{\piece 1234, \piece 1342, \piece 2143, \piece 3124\}$
\vspace{\smallskipamount}

{Seq= 4, 42, 816, 25520, 1170240, 73992912, 6169370368, 655847011584 }

{Var= 2, 14, 204, 5104, 195040, 10570416, 771171296, 72871890176, 8658173200896 }

OEIS:  \seqnum{A122647} [1] Number of permutations of length 2n-1 with no local maxima or minima in even positions.
\vspace{\medskipamount}

\goodbreak{\bf  AEFJ }\quad $\{\piece 1234, \piece 1423, \piece 1432, \piece 2314\}$
\vspace{\smallskipamount}

{Seq= 4, 8, 13, 19, 26, 34, 43, 53, 64, 76, 89 }

OEIS:  \seqnum{A034856} [1] C(n + 1, 2) + n - 1.
\vspace{\medskipamount}

\goodbreak{\bf  AEFL }\quad $\{\piece 1234, \piece 1423, \piece 1432, \piece 2413\}$
\quad See also 
{\bf AELQ}
\vspace{\smallskipamount}

{Seq= 4, 10, 19, 31, 46, 64, 85, 109, 136, 166, 199 }

OEIS:  \seqnum{A005448} [1] Centered triangular numbers: 3n(n-1)/2 + 1.
\vspace{\medskipamount}

\goodbreak{\bf  AEFQ }\quad $\{\piece 1234, \piece 1423, \piece 1432, \piece 3214\}$
\quad See also 
{\bf AEJL}
\vspace{\smallskipamount}

{Seq= 4, 9, 16, 25, 36, 49, 64, 81, 100, 121, 144 }

OEIS:  \seqnum{A000290} [9] The squares: a(n) = n\^{}2.
\vspace{\medskipamount}

\goodbreak{\bf  AEFS }\quad $\{\piece 1234, \piece 1423, \piece 1432, \piece 3412\}$
\quad See also 
{\bf AEQS}
\vspace{\smallskipamount}

{Seq= 4, 12, 46, 188, 774, 3174, 12954, 52668, 213526, 863822, 3488874 }

{Var= 2, 10, 44, 186, 772, 3172, 12952, 52666, 213524, 863820, 3488872 }

OEIS:  \seqnum{A068551} [1] 4\^{}n - binomial(2n,n).
\vspace{\medskipamount}

\goodbreak{\bf  AEFZ }\quad $\{\piece 1234, \piece 1423, \piece 1432, \piece 4321\}$
\quad See also 
{\bf AEQZ}
\vspace{\smallskipamount}

{Seq= 4, 11, 31, 85, 228, 604, 1590, 4173, 10937, 28647, 75014 }

{Var= 2, 9, 29, 83, 226, 602, 1588, 4171, 10935, 28645, 75012 }

OEIS:  \seqnum{A152891} [1] a=b=0;b(n)=b+n+a;a(n)=a+n+b.
\vspace{\medskipamount}

\goodbreak{\bf  AEJS }\quad $\{\piece 1234, \piece 1423, \piece 2314, \piece 3412\}$
\vspace{\smallskipamount}

{Seq= 4, 10, 32, 114, 422, 1586, 6008, 22882, 87518, 335922, 1293294 }

{Var= 2, 8, 30, 112, 420, 1584, 6006, 22880, 87516, 335920, 1293292 }

OEIS:  \seqnum{A162551} [1] 2 * C(2n,n-1).
\vspace{\medskipamount}

\goodbreak{\bf  AELM }\quad $\{\piece 1234, \piece 1423, \piece 2413, \piece 2431\}$
\vspace{\smallskipamount}

{Seq= 4, 10, 28, 82, 242, 710, 2064, 5946, 16992, 48222, 136034 }

{Var= 2, 5, 14, 41, 121, 355, 1032, 2973, 8496, 24111, 68017 }

OEIS:  \seqnum{A116845} [1] Number of permutations of length n which avoid the patterns 231, 12534.
\vspace{\medskipamount}

\goodbreak{\bf  AELW }\quad $\{\piece 1234, \piece 1423, \piece 2413, \piece 4213\}$
\vspace{\smallskipamount}

{Seq= 4, 12, 40, 132, 422, 1310, 3972, 11824, 34692, 100612, 289034 }

{Var= 3, 11, 39, 131, 421, 1309, 3971, 11823, 34691, 100611, 289033 }

OEIS:  \seqnum{A166336} [1] Expansion of (1-4x+7x\^{}2-4x\^{}3+x\^{}4)/(1-7x+17x\^{}2-17x\^{}3+7x\^{}4-x\^{}5)
\vspace{\medskipamount}

\goodbreak{\bf  AFLQ }\quad $\{\piece 1234, \piece 1432, \piece 2413, \piece 3214\}$
\quad See also 
{\bf AKPR}, {\bf AKPV}, {\bf APRX}, {\bf CDJQ}, {\bf CDLQ}, {\bf CEMQ}, {\bf CFMQ}, {\bf DEJQ}, {\bf DFJQ}
\vspace{\smallskipamount}

{Seq= 4, 11, 22, 37, 56, 79, 106, 137, 172, 211, 254 }

OEIS:  \seqnum{A084849} [1] 1+n+2n\^{}2.
\vspace{\medskipamount}

\goodbreak{\bf  AKMP }\quad $\{\piece 1234, \piece 2341, \piece 2431, \piece 3142\}$
\quad See also 
{\bf AMPX}
\vspace{\smallskipamount}

{Seq= 4, 13, 56, 241, 1000, 4061, 16336, 65473, 262064, 1048477, 4194184 }

{Var= 3, 12, 55, 240, 999, 4060, 16335, 65472, 262063, 1048476, 4194183 }

OEIS:  \seqnum{A024038} [1] 4\^{}n-n\^{}2.
\vspace{\medskipamount}

\goodbreak{\bf  AKMT }\quad $\{\piece 1234, \piece 2341, \piece 2431, \piece 3421\}$
\quad See also 
{\bf BDFM}
\vspace{\smallskipamount}

{Seq= 4, 9, 23, 65, 197, 626, 2056, 6918, 23714, 82500, 290512 }

OEIS:  \seqnum{A014137} [1] Partial sums of Catalan numbers (A000108).
\vspace{\medskipamount}

\goodbreak{\bf  AKMV }\quad $\{\piece 1234, \piece 2341, \piece 2431, \piece 4132\}$
\quad See also 
{\bf AKRS}, {\bf AKSV}, {\bf AMVX}, {\bf CDEW}, {\bf CDJW}, {\bf CDLW}, {\bf DEJW}
\vspace{\smallskipamount}

{Seq= 4, 10, 34, 130, 514, 2050, 8194, 32770, 131074, 524290, 2097154 }

{Var= 2, 5, 17, 65, 257, 1025, 4097, 16385, 65537, 262145, 1048577 }

OEIS:  \seqnum{A052539} [2] 4\^{}n + 1.
\vspace{\medskipamount}

\goodbreak{\bf  AKPU }\quad $\{\piece 1234, \piece 2341, \piece 3142, \piece 4123\}$
\quad See also 
{\bf AKPX}, {\bf CDEQ}, {\bf CDFQ}, {\bf DEFQ}
\vspace{\smallskipamount}

{Seq= 4, 9, 18, 31, 48, 69, 94, 123, 156, 193, 234 }

OEIS:  \seqnum{A100037} [1] Positions of occurrences of the natural numbers as second subsequence in A100035.
\vspace{\medskipamount}

\goodbreak{\bf  AKRV }\quad $\{\piece 1234, \piece 2341, \piece 3241, \piece 4132\}$
\quad See also 
{\bf ARVX}, {\bf CDJL}, {\bf CEJM}, {\bf CEQW}
\vspace{\smallskipamount}

{Seq= 4, 8, 12, 16, 20, 24, 28, 32, 36, 40, 44 }

OEIS:  \seqnum{A008586} [6] Multiples of 4.
\vspace{\medskipamount}

\goodbreak{\bf  AKRW }\quad $\{\piece 1234, \piece 2341, \piece 3241, \piece 4213\}$
\quad See also 
{\bf AMPR}, {\bf APRS}
\vspace{\smallskipamount}

{Seq= 4, 15, 62, 253, 1020, 4091, 16378, 65529, 262136, 1048567, 4194294 }

{Var= 3, 14, 61, 252, 1019, 4090, 16377, 65528, 262135, 1048566, 4194293 }

OEIS:  \seqnum{A024037} [1] 4\^{}n-n.
\vspace{\medskipamount}

\goodbreak{\bf  AKUX }\quad $\{\piece 1234, \piece 2341, \piece 4123, \piece 4231\}$
\quad See also 
{\bf CDEF}
\vspace{\smallskipamount}

{Seq= 4, 4, 4, 4, 4, 4, 4, 4, 4, 4, 4 }

OEIS:  \seqnum{A000523} [213] Log\_2(n) rounded down.
\vspace{\medskipamount}

\goodbreak{\bf  AKWY }\quad $\{\piece 1234, \piece 2341, \piece 4213, \piece 4312\}$
\quad See also 
{\bf BDQW}
\vspace{\smallskipamount}

{Seq= 4, 16, 92, 604, 4214, 30538, 227476, 1730788, 13393690, 105089230, 834086422 }

{Var= 3, 15, 91, 603, 4213, 30537, 227475, 1730787, 13393689, 105089229, 834086421 }

OEIS:  \seqnum{A099251} [1] Bisection of A005043.
\vspace{\medskipamount}

\goodbreak{\bf  APRV }\quad $\{\piece 1234, \piece 3142, \piece 3241, \piece 4132\}$
\quad See also 
{\bf CEFW}
\vspace{\smallskipamount}

{Seq= 4, 13, 26, 43, 64, 89, 118, 151, 188, 229, 274 }

OEIS:  \seqnum{A091823} [1] a(n) = 2*n\^{}2 + 3*n - 1.
\vspace{\medskipamount}

\goodbreak{\bf  APSX }\quad $\{\piece 1234, \piece 3142, \piece 3412, \piece 4231\}$
\quad See also 
{\bf DFQW}
\vspace{\smallskipamount}

{Seq= 4, 16, 64, 256, 1024, 4096, 16384, 65536, 262144, 1048576, 4194304 }

OEIS:  \seqnum{A000302} [2] Powers of 4.
\vspace{\medskipamount}

\goodbreak{\bf  ARSV }\quad $\{\piece 1234, \piece 3241, \piece 3412, \piece 4132\}$
\vspace{\smallskipamount}

{Seq= 4, 14, 60, 250, 1016, 4086, 16372, 65522, 262128, 1048558, 4194284 }

{Var= 2, 12, 58, 248, 1014, 4084, 16370, 65520, 262126, 1048556, 4194282 }

OEIS:  \seqnum{A100103} [1] 2\^{}(2*n)-(2*n).
\vspace{\medskipamount}

\goodbreak{\bf  BCEG }\quad $\{\piece 1243, \piece 1324, \piece 1423, \piece 2134\}$
\qquad See also 
{\bf BEGJ}, {\bf BEGY}, {\bf BEJV}, {\bf BEVY}, {\bf BJTV}, {\bf EJRV}
\vspace{\smallskipamount}

{Seq= 4, 24, 272, 4960, 132672, 4893056, 237969664, 14756156928, 1136284574720 }

{Var= 1, 3, 17, 155, 2073, 38227, 929569, 28820619, 1109652905, 51943281731 }

OEIS:  \seqnum{A110501} [2] Unsigned Genocchi numbers (of first kind) of even index.
\vspace{\medskipamount}

\goodbreak{\bf  BCEK }\quad $\{\piece 1243, \piece 1324, \piece 1423, \piece 2341\}$
\vspace{\smallskipamount}

{Seq= 4, 14, 62, 324, 1936, 12962, 95786, 772196, 6729124, 62920648, 627487330 }

{Var= 2, 7, 31, 162, 968, 6481, 47893, 386098, 3364562, 31460324, 313743665 }

OEIS:  \seqnum{A125275} [1] Eigensequence of triangle A039599: a(n) = Sum\_\{k=0..n-1\} A039599(n-1,k)*a(k) for n$>$0 with a(0)=1.
\vspace{\medskipamount}

\goodbreak{\bf  BDEF }\quad $\{\piece 1243, \piece 1342, \piece 1423, \piece 1432\}$
\quad See also 
{\bf BDEM}, {\bf BDFT}, {\bf BEFK}, {\bf BEKM}, {\bf BFKT}, {\bf DEFR}, {\bf DEMR}, {\bf EFKR}
\vspace{\smallskipamount}

{Seq= 4, 10, 28, 84, 264, 858, 2860, 9724, 33592, 117572, 416024 }

OEIS:  \seqnum{A068875} [2] Expansion of (1+x*C)*C, where C = (1-(1-4*x)\^{}(1/2))/(2*x) is g.f. for Catalan numbers, A000108.
\vspace{\medskipamount}

\goodbreak{\bf  BDEK }\quad $\{\piece 1243, \piece 1342, \piece 1423, \piece 2341\}$
\vspace{\smallskipamount}

{Seq= 4, 12, 38, 126, 430, 1498, 5300, 18980, 68636, 250208, 918304 }

{Var= 2, 6, 19, 63, 215, 749, 2650, 9490, 34318, 125104, 459152 }

OEIS:  \seqnum{A109262} [1] A Catalan transform of the Fibonacci numbers.
\vspace{\medskipamount}

\goodbreak{\bf  BDFK }\quad $\{\piece 1243, \piece 1342, \piece 1432, \piece 2341\}$
\vspace{\smallskipamount}

{Seq= 4, 11, 32, 99, 318, 1051, 3550, 12200, 42520, 149930, 533890 }

OEIS:  \seqnum{A135339} [1] Number of Dyck paths of semilength n having no DUDU's starting at level 1.
\vspace{\medskipamount}

\goodbreak{\bf  BEFM }\quad $\{\piece 1243, \piece 1423, \piece 1432, \piece 2431\}$
\vspace{\smallskipamount}

{Seq= 4, 8, 18, 46, 130, 394, 1252, 4112, 13836, 47428, 165000 }

{Var= 2, 4, 9, 23, 65, 197, 626, 2056, 6918, 23714, 82500 }

OEIS:  \seqnum{A014137} [1] Partial sums of Catalan numbers (A000108).
\vspace{\medskipamount}

\goodbreak{\bf  BEFT }\quad $\{\piece 1243, \piece 1423, \piece 1432, \piece 3421\}$
\vspace{\smallskipamount}

{Seq= 4, 10, 30, 98, 336, 1188, 4290, 15730, 58344, 218348, 823004 }

{Var= 1, 2, 5, 14, 42, 132, 429, 1430, 4862, 16796, 58786 }

OEIS:  \seqnum{A000108} [5] Catalan numbers: C(n) = binomial(2n,n)/(n+1) = (2n)!/(n!(n+1)!). Also called Segner numbers.
\vspace{\medskipamount}

\goodbreak{\bf  BERT }\quad $\{\piece 1243, \piece 1423, \piece 3241, \piece 3421\}$
\vspace{\smallskipamount}

{Seq= 4, 16, 80, 448, 2688, 16896, 109824, 732160, 4978688, 34398208, 240787456 }

OEIS:  \seqnum{A025225} [2] a(n) = a(1)*a(n-1) + a(2)*a(n-2) + ...+ a(n-1)*a(1) for n $\geq$ 2. Also a(n) = (2\^{}n)*C(n-1), where C = A000108 (Catalan numbers).
\vspace{\medskipamount}

\goodbreak{\bf  BFKM }\quad $\{\piece 1243, \piece 1432, \piece 2341, \piece 2431\}$
\vspace{\smallskipamount}

{Seq= 4, 9, 22, 58, 163, 483, 1494, 4783, 15740, 52956, 181391 }

OEIS:  \seqnum{A059019} [1] Number of Dyck paths of semilength n with no peak at height 3.
\vspace{\medskipamount}

\goodbreak{\bf  BFLM }\quad $\{\piece 1243, \piece 1432, \piece 2413, \piece 2431\}$
\vspace{\smallskipamount}

{Seq= 4, 8, 20, 60, 200, 704, 2552, 9416, 35156, 132396, 501908 }

{Var= 1, 2, 5, 15, 50, 176, 638, 2354, 8789, 33099, 125477 }

OEIS:  \seqnum{A024718} [1] (1/2)*(1 + sum of C(2k,k)) for k = 0,1,2,...,n.
\vspace{\medskipamount}

\goodbreak{\bf  CDEK }\quad $\{\piece 1324, \piece 1342, \piece 1423, \piece 2341\}$
\quad See also 
{\bf CDFK}
\vspace{\smallskipamount}

{Seq= 4, 10, 32, 106, 412, 1634, 7240, 32722, 160436, 803002, 4279024 }

{Var= 2, 5, 16, 53, 206, 817, 3620, 16361, 80218, 401501, 2139512 }

OEIS:  \seqnum{A081126} [1] Binomial transform of n!/floor(n/2)!.
\vspace{\medskipamount}

\goodbreak{\bf  CDFW }\quad $\{\piece 1324, \piece 1342, \piece 1432, \piece 4213\}$
\quad See also 
{\bf CDQW}, {\bf DEFW}, {\bf DEQW}
\vspace{\smallskipamount}

{Seq= 4, 13, 49, 193, 769, 3073, 12289, 49153, 196609, 786433, 3145729 }

OEIS:  \seqnum{A140660} [1] 3*4\^{}n+1.
\vspace{\medskipamount}

\goodbreak{\bf  CDJM }\quad $\{\piece 1324, \piece 1342, \piece 2314, \piece 2431\}$
\vspace{\smallskipamount}

{Seq= 4, 12, 30, 62, 112, 184, 282, 410, 572, 772, 1014 }

{Var= 2, 6, 15, 31, 56, 92, 141, 205, 286, 386, 507 }

OEIS:  \seqnum{A056520} [1] (n+2)*(2*n\^{}2-n+3)/6
\vspace{\medskipamount}

\goodbreak{\bf  CDKM }\quad $\{\piece 1324, \piece 1342, \piece 2341, \piece 2431\}$
\vspace{\smallskipamount}

{Seq= 4, 12, 40, 152, 624, 2768, 13024, 64800, 337984, 1842368, 10444416 }

{Var= 2, 6, 20, 76, 312, 1384, 6512, 32400, 168992, 921184, 5222208 }

OEIS:  \seqnum{A000898} [1] a(n) = 2(a(n-1) + (n-1)a(n-2)).
\vspace{\medskipamount}

\goodbreak{\bf  CDLM }\quad $\{\piece 1324, \piece 1342, \piece 2413, \piece 2431\}$
\vspace{\smallskipamount}

{Seq= 4, 10, 20, 34, 52, 74, 100, 130, 164, 202, 244 }

OEIS:  \seqnum{A005893} [1] Number of points on surface of tetrahedron: 2n\^{}2 + 2 (coordination sequence for sodalite net) for n$>$0.
\vspace{\medskipamount}

\goodbreak{\bf  CMPW }\quad $\{\piece 1324, \piece 2431, \piece 3142, \piece 4213\}$
\vspace{\smallskipamount}

{Seq= 4, 36, 568, 14560, 546492, 28289184, 1930982576, 168054225408, 18162775533620 }

{Var= 2, 18, 284, 7280, 273246, 14144592, 965491288, 84027112704, 9081387766810 }

OEIS:  \seqnum{A131455} [1] Number of inequivalent properly oriented and labeled planar chord diagrams whose associated planar tree is a path on n+1 ver...
\vspace{\medskipamount}

\goodbreak{\bf  DEJM }\quad $\{\piece 1342, \piece 1423, \piece 2314, \piece 2431\}$
\quad See also 
{\bf DFJM}
\vspace{\smallskipamount}

{Seq= 4, 10, 24, 50, 92, 154, 240, 354, 500, 682, 904 }

{Var= 2, 5, 12, 25, 46, 77, 120, 177, 250, 341, 452 }

OEIS:  \seqnum{A116731} [1] Number of permutations of length n which avoid the patterns 321, 2143, 3124; or avoid the patterns 132, 2314, 4312, etc.
}



\bigskip
\hrule
\bigskip

\noindent 2010 {\it Mathematics Subject Classification}:
Primary 05A15, 05A19, 05A05; Secondary 05B50.


\noindent \emph{Keywords: } standard puzzles,   integer sequences,
	Fibonacci numbers, tangent numbers, Catalan numbers.

\bigskip
\hrule
\bigskip

\noindent {\small (Concerned with sequences
\seqnum{A000012}, \seqnum{A000108}, \seqnum{A121373}, \seqnum{A000165}, \seqnum{A055642}, \seqnum{A125273}, \seqnum{A000027}, \seqnum{A000124},
\seqnum{A002522}, \seqnum{A130883}, \seqnum{A005843}, \seqnum{A000108}, \seqnum{A068875}, \seqnum{A005807}, \seqnum{A117513}, \seqnum{A121686},
\seqnum{A028329}, \seqnum{A052676}, \seqnum{A001147}, \seqnum{A128135}, \seqnum{A000027}, \seqnum{A022856}, \seqnum{A059100}, \seqnum{A032908},
\seqnum{A115112}, \seqnum{A104858}, \seqnum{A000045}, \seqnum{A000217}, \seqnum{A005408}, \seqnum{A002061}, \seqnum{A035508}, \seqnum{A001791},
\seqnum{A054444}, \seqnum{A114185}, \seqnum{A104249}, \seqnum{A008549}, \seqnum{A033816}, \seqnum{A155587}, \seqnum{A010701}, \seqnum{A084508},
\seqnum{A124133}, \seqnum{A014105}, \seqnum{A038665}, \seqnum{A004767}, \seqnum{A000245}, \seqnum{A005807}, \seqnum{A000782}, \seqnum{A001700},
\seqnum{A081293}, \seqnum{A078718}, \seqnum{A120589}, \seqnum{A079309}, \seqnum{A178789}, \seqnum{A179903}, \seqnum{A029744}, \seqnum{A051577},
\seqnum{A032184}, \seqnum{A097801}, \seqnum{A010844}, \seqnum{A002866}, \seqnum{A051577}, \seqnum{A004400}, \seqnum{A052676}, \seqnum{A125273},
\seqnum{A000124}, \seqnum{A005843}, \seqnum{A014206}, \seqnum{A055588}, \seqnum{A112849}, \seqnum{A176006}, \seqnum{A027989}, \seqnum{A140090},
\seqnum{A052995}, \seqnum{A100320}, \seqnum{A001003}, \seqnum{A000079}, \seqnum{A122647}, \seqnum{A034856}, \seqnum{A005448}, \seqnum{A000290},
\seqnum{A068551}, \seqnum{A152891}, \seqnum{A162551}, \seqnum{A116845}, \seqnum{A166336}, \seqnum{A084849}, \seqnum{A024038}, \seqnum{A014137},
\seqnum{A052539}, \seqnum{A100037}, \seqnum{A008586}, \seqnum{A024037}, \seqnum{A000523}, \seqnum{A099251}, \seqnum{A091823}, \seqnum{A000302},
\seqnum{A100103}, \seqnum{A110501}, \seqnum{A125275}, \seqnum{A068875}, \seqnum{A109262}, \seqnum{A135339}, \seqnum{A014137}, \seqnum{A000108},
\seqnum{A025225}, \seqnum{A059019}, \seqnum{A024718}, \seqnum{A081126}, \seqnum{A140660}, \seqnum{A056520}, \seqnum{A000898}, \seqnum{A005893},
\seqnum{A131455} and \seqnum{A116731}.)
}

\bigskip
\hrule
\bigskip

\begin{thebibliography}{99}

\bibitem[Du]{Dumont}
Dumont, Dominique, Interpr\'etations
combinatoires des nombres de Genocchi, Duke math. J., {\bf 41(2)} (1974), 
pp. 305--318.

\bibitem[FHa]{FHa}
Foata, Dominique;  Han, Guo-Niu, Doubloons and new $q$-tangent numbers, Quarterly Journal of Mathematics, in press, 2009, 17 pages.

\bibitem[FHb]{FHb}
Foata, Dominique;  Han, Guo-Niu, Doubloons and $q$-secant numbers, Munster J. of Math., 
{\bf 3} (2010), pp. 89--110.

\bibitem[FRT]{Hook}
Frame, J. Sutherland; Robinson, Gilbert de Beauregard;
Thrall, Robert M., The hook graphs of the symmetric groups, Canadian
J. Math., {\bf 6} (1954), pp. 316--324.

\bibitem[GZ]{GZ}
Graham, Ron; Zang, Nan,
Enumerating split-pair arrangements,
J. Combin. Theory, Ser.~A, {\bf 115} (2008), p. 293--303.



\bibitem[S]{Sloane}
Sloane, N. J. A., The On-Line Encyclopedia of Integer Sequences. Published electronically at \href{http://oeis.org/}{\tt http://oeis.org/}.

\bibitem[St]{Stanley}
Stanley, Richard P., Enumerative Combinatorics, vol. 2,
Cambridge University Press, {1999}.

\bibitem[WikiC]{Wiki_Catalan}
Wikipedia: \href{http://en.wikipedia.org/wiki/Catalan_number}{Catalan number}.

\bibitem[WikiP]{Wiki_Polyomino}
Wikipedia \href{http://en.wikipedia.org/wiki/Polyomino}{Polyomino}.

\bibitem[WikiPP]{Wiki_Pattern}
Wikipedia: \href{http://en.wikipedia.org/wiki/Permutation_pattern}{Permutation pattern}.

\bibitem[WikiT]{Wiki_Tangent}
Wikipedia: \href{http://en.wikipedia.org/wiki/Tangent_number}{Alternating permutation}.

\bibitem[WikiY]{Wiki_YoungTab}
Wikipedia: \href{http://en.wikipedia.org/wiki/Young_tableaux}{Young tableaux}.


\end{thebibliography}
\end{document}